\documentclass[11pt]{article}

\usepackage[a4paper,margin=1in]{geometry}
\usepackage{amsmath,amssymb,amsthm,mathtools}
\usepackage{bm}
\usepackage{enumitem}
\usepackage{booktabs}
\usepackage{algorithm}
\usepackage{float}
\usepackage{graphicx}
\usepackage{algpseudocode}
\usepackage{xcolor}
\usepackage{microtype}
\usepackage[T1]{fontenc}
\usepackage{lmodern}
\usepackage[colorlinks=true,linkcolor=blue!55!black,citecolor=blue!55!black,urlcolor=blue!55!black]{hyperref}

\allowdisplaybreaks
\numberwithin{equation}{section}

\newcommand{\R}{\mathbb{R}}
\newcommand{\C}{\mathbb{C}}

\newcommand{\Sph}{\mathbb{S}}
\newcommand{\ii}{\mathrm{i}}
\newcommand{\dd}{\,\mathrm d}
\newcommand{\supp}{\operatorname{supp}}
\newcommand{\rank}{\operatorname{rank}}

\newcommand{\curl}{\nabla\times}

\newcommand{\vect}[1]{\bm{#1}}
\newcommand{\tens}[1]{\bm{\mathsf{#1}}}

\newcommand{\vecop}{\ensuremath{\operatorname{vec}}}
\newcommand{\Id}{\ensuremath{\tens{I}}}

\newcommand{\F}{\mathcal{F}}
\newcommand{\Dop}{\mathcal{D}}
\newcommand{\G}{\ensuremath{\tens{G}}}

\newcommand{\epsr}{\ensuremath{\bm{\varepsilon}_{\mathrm r}}}
\newcommand{\Qt}{\ensuremath{\tens{Q}}}
\newcommand{\Tt}{\ensuremath{\tens{T}}}
\newcommand{\Mt}{\ensuremath{\tens{M}}}
\newcommand{\Bt}{\ensuremath{\tens{B}}}

\newcommand{\Frob}{\mathrm F}

\theoremstyle{plain}
\newtheorem{theorem}{Theorem}[section]
\newtheorem{proposition}[theorem]{Proposition}
\newtheorem{lemma}[theorem]{Lemma}
\newtheorem{corollary}[theorem]{Corollary}
\theoremstyle{definition}

\newtheorem{assumptionx}[theorem]{Assumption}
\newenvironment{keywords}{\par\noindent\textbf{Keywords.} }{\par}
\newenvironment{MSCcodes}{\par\noindent\textbf{MSC codes.} }{\par}
\newcommand{\email}[1]{\texttt{#1}}

\title{Polarimetric Low-Rank Imaging of Anisotropic Maxwell Media}
\author{Xianbo Chen\thanks{School of Mathematical Sciences, Zhejiang University, Hangzhou 310058, China (\email{chenxianbo@zju.edu.cn}).}
\and Jingzhi Li\thanks{Department of Mathematics \& SUSTech International Center for Mathematics, Southern University of Science and Technology, Shenzhen 518052, China (\email{li.jz@sustech.edu.cn}).}
\and Ping Liu\thanks{School of Mathematical Sciences and Institute of Fundamental and Transdisciplinary Research, Zhejiang University, Hangzhou 310058, China (\email{pingliu@zju.edu.cn}).}}
\date{}

\begin{document}
\maketitle

\begin{abstract}
We consider three-dimensional imaging of compactly supported anisotropic electric media from fixed-frequency polarimetric Maxwell far-field data. Maxwell transversality creates an obstruction absent from scalar Born inversion: each measurement determines only a projected action of the tensor-valued Fourier transform. We first resolve this obstruction by a finite-dimensional polarimetric inversion. For a general complex $3\times3$ contrast, three incidence--observation configurations are necessary, and we give an explicit three-configuration design that is sufficient at every nonzero interior Fourier point. Its Gram determinant yields quantitative conditioning information and exhibits the loss of recoverability at the boundary of the accessible Fourier ball. After tensor demixing, the inverse problem reduces componentwise to restricted Fourier inversion on the unit ball, which we regularize by spectral cutoff in three-dimensional ball generalized prolate spheroidal wave functions. For discrete acquisition, a combined reconstruction error estimate separates the effects of measurement noise, polarimetric amplification, off-grid Fourier mismatch, quadrature error, Born modeling error, and spectral truncation. Numerical tests examine polarimetric conditioning and use analytically generated Born data to compare three- and six-configuration recovery and assess off-grid sensitivity. Reconstructions from full-Maxwell volume-integral data assess the effects of noise, wavenumber, and basis choice, as well as the recovery of different tensor components.
\end{abstract}

\begin{keywords}
inverse electromagnetic scattering, anisotropic permittivity, polarimetric imaging, restricted Fourier transform, generalized prolate spheroidal wave functions, low-rank regularization
\end{keywords}

\begin{MSCcodes}
35R30, 65R32, 78A46
\end{MSCcodes}

\section{Introduction}\label{sec:intro}

Three-dimensional electromagnetic imaging of anisotropic media contains an information bottleneck that is absent from scalar inverse scattering. Under a first Born linearization, scalar far-field data can be related directly to Fourier samples of the material contrast. In the Maxwell setting, however, the unknown electric contrast is tensor valued, while both the incident polarization and the measured electric far field are constrained by transversality. Consequently, a single polarimetric measurement reveals only a projected action of the tensor-valued Fourier transform rather than the tensor itself. A quantitative reconstruction method must therefore first resolve a finite-dimensional tensor-identification problem at each accessible Fourier point and subsequently invert the resulting spatially band-limited Fourier data. The present paper separates these two inverse mechanisms and tracks the stability of the polarimetric stage through the subsequent regularized reconstruction.

Inverse electromagnetic scattering has been studied through qualitative and quantitative approaches; standard references include \cite{CakoniColtonHaddar2016,CakoniColtonMonk2011,ColtonKress2019,KirschGrinberg2008}. Sampling and factorization techniques can localize targets efficiently \cite{HarrisNguyen2020,ItoJinZou2012,LeNguyenSchmidtTruong2022,LiLiuWang2013,Potthast2010}, whereas contrast-source and related iterative methods seek quantitative material reconstructions from the full nonlinear model \cite{LiAbubakarVanDenBerg2009}. For anisotropic electromagnetics, uniqueness theory is available under structural assumptions \cite{KenigSaloUhlmann2011}, and polarization-sensitive linearized inverse problems have been analyzed in several settings. Gilman, Smith, and Tsynkov \cite{GilmanSmithTsynkov2012} derived explicit relations between polarized scattering data and the parameters of a homogeneous uniaxial medium under Born linearization. Elbau, Mindrinos, and Scherzer \cite{ElbauMindrinosScherzer2015} combined incident polarizations and sample rotations to establish recovery conditions for general tensor-valued Radon data under the first-order Born approximation, without carrying out a numerical reconstruction of the spatial tensor. Their subsequent work \cite{ElbauMindrinosScherzer2018} formulated integral equations for an orthotropic susceptibility tensor using a second-order Born approximation. Balandin \cite{Balandin2023,BalandinErratum2024} derived a tensor Fourier diffraction relation connecting planar scattered-field measurements to the Fourier transform of the dielectric contrast. The transverse projection in this relation is singular, however, and cannot be removed by ordinary matrix inversion. These studies either focus on restricted tensor classes, such as uniaxial or orthotropic media, or establish recovery relations for general tensors without quantifying the conditioning and stability of the polarimetric inversion. They also do not combine this recovery with a basis projection adapted to the restricted Fourier operator to quantitatively reconstruct the spatially varying material tensor.

A second ingredient is the spectral inversion of Fourier data known only on a bounded region. Generalized prolate spheroidal wave functions (GPSWFs) are eigenfunctions of the restricted Fourier transform and are therefore adapted simultaneously to spatial support and Fourier bandwidth \cite{Greengard2024,Slepian1964,ZhangLiWangZhang2020}. For scalar inverse scattering, Meng \cite{Meng2023} developed a Picard regularization framework, and Zhou, Audibert, Meng, and Zhang \cite{ZhouAudibertMengZhang2025} exploited GPSWF low-rank structure numerically and analytically. Our use of ball GPSWFs builds on this scalar restricted-Fourier theory. The Maxwell contribution is not the GPSWF basis by itself: before restricted-Fourier inversion is available, one must demix a full tensor from transverse polarimetric measurements, control the geometry-dependent conditioning, and account for the fact that a finite set of realizable incidence--observation pairs generally samples nearby rather than identical Fourier nodes.

The main results are as follows. First, we extend the well-posedness result for the forward problem to general bounded tensor contrasts in strictly dissipative media whose relative permittivity has a uniformly positive definite real part, and derive a volume-integral representation of the scattered field. Using the first-order Born approximation, we express the far field in terms of a transverse projection of the tensor Fourier transform acting on the incident polarization. Second, we use this relation to construct a finite-dimensional polarimetric recovery method and analyze its identifiability and conditioning. For a general complex three-by-three tensor, a single incidence--observation configuration carries at most four independent complex degrees of freedom, so at least three configurations are necessary. For every nonzero interior Fourier point, we construct three equiangular configurations and prove that they are sufficient. A closed-form determinant of the associated Gram matrix establishes full rank at these points, and singular-value asymptotics quantify the deterioration of this design near the center and boundary. At the center, a separate three-axis design is full rank; at the boundary, a general tensor is not identifiable. Third, after polarimetric recovery, we regularize the restricted Fourier inversion on the ball by GPSWF spectral cutoff to reconstruct the spatially varying tensor contrast. Finally, for discrete acquisition, we prove a Lipschitz estimate for the error caused by replacing a target quadrature node by nearby realizable Fourier points and derive a combined reconstruction error bound. This bound accounts for measurement noise, Born modeling error, off-grid mismatch, quadrature error, and spectral truncation, together with polarimetric and spectral amplification.

The numerical experiments examine both polarimetric recovery and spatial reconstruction. Using analytically generated Born data, we compare tensor recovery from three and six measurement configurations in the presence of noise. We then test how the recovery error changes with the distance between the target Fourier nodes and the nearby measurement points, and compare the results with the off-grid error estimate. For spatial reconstruction, we use far-field data generated by the full Maxwell volume integral equation and apply the Born-based reconstruction. We examine the effects of the retained GPSWF dimension, measurement noise, and wavenumber, and compare GPSWF, Fourier, and spherical-harmonic Bessel bases using the same recovered Fourier data. Finally, we reconstruct different components of the complex tensor contrast to examine the recovery of the material differences among the three inclusions.

The paper is organized as follows. Section~\ref{sec:forward} derives the Maxwell far-field and tensor Fourier relation and quantifies the modeling remainder introduced by the Born linearization. Section~\ref{sec:born} establishes polarimetric identifiability and conditioning. Section~\ref{sec:gpswf} develops the restricted-Fourier GPSWF inversion. Section~\ref{sec:discretization} analyzes off-grid acquisition and derives a combined discrete reconstruction error estimate. Section~\ref{sec:numerics} reports numerical results. The appendices include implementation details for the ball GPSWF basis.

\section{Maxwell model and tensor diffraction relation}\label{sec:forward}

Let $D\subset\R^3$ be a bounded open set with Lipschitz boundary occupied by the inhomogeneous medium. We consider a homogeneous, nonmagnetic background with wavenumber $k>0$ and use the time dependence $e^{-\ii\omega t}$. The relative permittivity is $\epsr=\Id+\Qt$, where $\Qt\in L^\infty(\R^3;\C^{3\times3})$ is supported in $\overline D$. The incident electric field is the plane wave $\vect{E}^{\rm i}(\vect{x})=\vect{e}e^{\ii k\vect{d}\cdot\vect{x}}$, with $\vect{d}\in\Sph^2$ and $\vect{e}\in\C^3$ satisfying $\vect{d}\cdot\vect{e}=0$. The total field $\vect{E}=\vect{E}^{\rm i}+\vect{E}^{\rm s}$ solves
\begin{equation}\label{eq:maxwell_total}
 \curl\curl\vect{E}-k^2(\Id+\Qt)\vect{E}=0\quad\text{in }\R^3,
\end{equation}
and the scattered field satisfies the Silver--M\"uller radiation condition
\begin{equation*}\label{eq:radiation}
 \lim_{r\to\infty}r\big[(\curl\vect{E}^{\rm s})(\vect{x})\times\hat{\vect{x}}
       -\ii k\vect{E}^{\rm s}(\vect{x})\big]=0,
 \quad r=|\vect{x}|,\quad\hat{\vect{x}}=\vect{x}/r,
\end{equation*}
uniformly in $\hat{\vect{x}}\in\Sph^2$.

To specify the class of anisotropic media considered here, we impose the following assumptions on the material contrast.
\begin{assumptionx}\label{ass:material}
There is a constant $\gamma>0$ such that, for almost every $\vect{x}\in D$,
\begin{equation*}
\operatorname{Re}\!\left[\vect{z}^*(\Id+\Qt(\vect{x}))\vect{z}\right]\geq\gamma|\vect{z}|^2
\quad\text{for all }\vect{z}\in\C^3,
\end{equation*}
and
\begin{equation}\label{equ:assumpq1}
\operatorname{Im}\!\left[\vect{z}^*(\Id+\Qt(\vect{x}))\vect{z}\right]\geq 0
\quad\text{for all }\vect{z}\in\C^3\setminus \{0\}.
\end{equation}
\end{assumptionx}

We next derive an integral representation of the scattered field. Since the background medium is homogeneous and isotropic, we begin with the outgoing fundamental solution of the scalar Helmholtz equation,
\begin{equation*}
\Phi_k(\vect{x})=\frac{e^{\ii k|\vect{x}|}}{4\pi|\vect{x}|},\quad \vect{x}\in\R^3\setminus\{\vect{0}\},
\end{equation*}
which satisfies
\begin{equation*}
(\Delta+k^2)\Phi_k=-\delta_0
\end{equation*}
in the distributional sense. The outgoing Green tensor associated with the Maxwell operator is defined distributionally by
\begin{equation*}\label{eq:maxwell_green_tensor}
\G_k(\vect{x})=\left(\Id+k^{-2}\nabla\nabla^\top\right)\Phi_k(\vect{x}).
\end{equation*}
Using the identity $\curl\curl=\nabla\nabla\cdot-\Delta$, one obtains
\begin{equation*}
(\curl\curl-k^2\Id)\G_k=\Id\delta_0
\end{equation*}
in the distributional sense.

\begin{theorem}\label{prop:volume_integral}
Let $\Qt$ be piecewise smooth and symmetric, satisfying Assumption~\ref{ass:material}. Then \eqref{eq:maxwell_total} admits a unique total field $\vect{E}\in H_{\mathrm{loc}}(\operatorname{curl};\R^3)$ whose scattered field satisfies the radiation condition. The scattered field can be represented by
\begin{equation}\label{eq:volume_representation}
\vect{E}^{\mathrm s}(\vect{x})=k^2\int_D\G_k(\vect{x}-\vect{y})\Qt(\vect{y})\vect{E}(\vect{y})\dd\vect{y},
\end{equation}
and consequently
\begin{equation}\label{eq:volume_integral_equation}
\vect{E}(\vect{x})=\vect{E}^{\mathrm i}(\vect{x})+k^2\int_D\G_k(\vect{x}-\vect{y})\Qt(\vect{y})\vect{E}(\vect{y})\dd\vect{y}.
\end{equation}
Moreover, for strictly dissipative media, i.e., when the inequality in \eqref{equ:assumpq1} is strict, the above conclusions hold for general $\Qt\in L^\infty(\R^3;\C^{3\times3})$.
\end{theorem}

\begin{proof}
For piecewise smooth symmetric contrasts, well-posedness follows from \cite[Theorem~4.2]{CakoniColtonMonk2011}. We now prove the extension to general bounded tensor contrasts under strict dissipativity.

Any radiating solution satisfies
\begin{equation*}
(\curl\curl-k^2\Id)\vect{E}^{\mathrm s}=k^2\Qt\vect{E}.
\end{equation*}
The outgoing Green tensor gives \eqref{eq:volume_representation} and hence \eqref{eq:volume_integral_equation}. Writing
\begin{equation*}
(V_Q\vect{u})(\vect{x})=\int_D\G_k(\vect{x}-\vect{y})\Qt(\vect{y})\vect{u}(\vect{y})\dd\vect{y},
\end{equation*}
the volume equation becomes $(I-k^2V_Q)\vect{E}=\vect{E}^{\mathrm i}$ in $L^2(D)^3$.

To establish existence and uniqueness, we first show that $I-k^2V_Q$ is Fredholm of index zero. Let $\Phi_0(\vect{x})=(4\pi|\vect{x}|)^{-1}$ and define
\begin{equation*}
\mathcal P\vect{f}=-\nabla\nabla^\top(\Phi_0*\vect{f}).
\end{equation*}
The derivatives are interpreted distributionally. 
In Fourier space,
\begin{equation*}
\widehat{\mathcal P\vect{f}}(\vect{\xi})
=\frac{\vect{\xi}\vect{\xi}^\top}{|\vect{\xi}|^2}\widehat{\vect{f}}(\vect{\xi}),\qquad \vect{\xi}\neq\vect{0}.
\end{equation*}
The multiplier is Hermitian and idempotent, which gives $\mathcal P^*=\mathcal P$ and $\mathcal P^2=\mathcal P$, so $\mathcal P$ is an orthogonal projection on $L^2(\R^3)^3$.
Products involving $\Qt$ are extended by zero outside $D$. Since
\begin{equation*}
k^2\G_k=k^2\Phi_k\Id+\nabla\nabla^\top\Phi_0+\nabla\nabla^\top(\Phi_k-\Phi_0),
\end{equation*}
we have
\begin{align}
(I-k^2V_Q)\vect{u}
&=\vect{u}-k^2\int_D\Phi_k(\vect{x}-\vect{y})\Qt(\vect{y})\vect{u}(\vect{y})\dd\vect{y}-\nabla_{\vect{x}}\nabla_{\vect{x}}^\top\int_D\Phi_k(\vect{x}-\vect{y})\Qt(\vect{y})\vect{u}(\vect{y})\dd\vect{y}\nonumber\\
&=\vect{u}-\nabla_{\vect{x}}\nabla_{\vect{x}}^\top\int_D\Phi_0(\vect{x}-\vect{y})\Qt(\vect{y})\vect{u}(\vect{y})\dd\vect{y}\nonumber-k^2\int_D\Phi_k(\vect{x}-\vect{y})\Qt(\vect{y})\vect{u}(\vect{y})\dd\vect{y}\\
&\quad-\nabla_{\vect{x}}\nabla_{\vect{x}}^\top\int_D(\Phi_k-\Phi_0)(\vect{x}-\vect{y})\Qt(\vect{y})\vect{u}(\vect{y})\dd\vect{y}\nonumber\\
&=\vect{u}+\left.\mathcal P(\Qt\vect{u})\right|_D+K\vect{u},
\label{equ:operatordecom1}
\end{align}
where
\begin{equation*}
(K\vect{u})(\vect{x})=-k^2\int_D\Phi_k(\vect{x}-\vect{y})\Qt(\vect{y})\vect{u}(\vect{y})\dd\vect{y}-\int_D\nabla_{\vect{x}}\nabla_{\vect{x}}^\top(\Phi_k-\Phi_0)(\vect{x}-\vect{y})\Qt(\vect{y})\vect{u}(\vect{y})\dd\vect{y}.
\end{equation*}
The operator $K$ is compact: both kernels are $O(|\vect{x}-\vect{y}|^{-1})$ and belong to $L^2(D\times D)$.

For $\vect{f}\in L^2(D)^3$, seek $\vect{w}\in\operatorname{Ran}\mathcal P$ such that
\begin{equation*}
\langle\vect{w},\vect{z}\rangle_{\R^3}+\langle\Qt\vect{w},\vect{z}\rangle_D
=\langle\Qt\vect{f},\vect{z}\rangle_D,
\end{equation*}
for all $\vect{z}\in\operatorname{Ran}\mathcal P$.
With inner products linear in the first argument, Assumption~\ref{ass:material} gives
\begin{equation*}
\begin{aligned}
\operatorname{Re}\left[\langle\vect{w},\vect{w}\rangle_{\R^3}+\langle\Qt\vect{w},\vect{w}\rangle_D\right]
&=\|\vect{w}\|_{L^2(\R^3\setminus D)^3}^2+\int_D\operatorname{Re}\left[\vect{w}^*(\Id+\Qt)\vect{w}\right]\dd\vect{x}\\
&\geq\|\vect{w}\|_{L^2(\R^3\setminus D)^3}^2+\gamma\|\vect{w}\|_{L^2(D)^3}^2\\
&\geq\min\{1,\gamma\}\|\vect{w}\|_{L^2(\R^3)^3}^2.
\end{aligned}
\end{equation*}
Lax--Milgram gives a unique $\vect{w}$. Because $\mathcal P$ is an orthogonal projection, its equation is
\begin{equation*}
\vect{w}+\mathcal P(\Qt\vect{w})=\mathcal P(\Qt\vect{f}).
\end{equation*}
Thus $\vect{u}=\vect{f}-\vect{w}|_D$ uniquely solves
\begin{equation*}
\vect{u}+\left.\mathcal P(\Qt\vect{u})\right|_D=\vect{f},
\end{equation*}
with $\|\vect{u}\|_{L^2(D)^3}\leq C\|\vect{f}\|_{L^2(D)^3}$. Consequently, together with decomposition (\ref{equ:operatordecom1}), $I-k^2V_Q$ is Fredholm of index zero.

If $(I-k^2V_Q)\vect{u}=0$, define
\begin{equation*}
\vect{E}(\vect{x})=k^2\int_D\G_k(\vect{x}-\vect{y})\Qt(\vect{y})\vect{u}(\vect{y})\dd\vect{y}.
\end{equation*}
Then $\vect{E}|_D=\vect{u}$ and $\curl\curl\vect{E}=k^2(\Id+\Qt)\vect{E}$. Applying the weak Green identity over a ball $B_\rho$ containing $\overline D$ gives
\begin{equation*}
\int_{B_\rho}\left(|\curl\vect{E}|^2-k^2\vect{E}^*(\Id+\Qt)\vect{E}\right)\dd\vect{x}
=\int_{\partial B_\rho}(\curl\vect{E}\times\vect{n})\cdot\overline{\vect{E}}\dd S.
\end{equation*}
Taking imaginary parts yields
\begin{equation*}
k^2\int_D\operatorname{Im}\!\left[\vect{E}^*(\Id+\Qt)\vect{E}\right]\dd\vect{x}
+\operatorname{Im}\int_{\partial B_\rho}(\curl\vect{E}\times\vect{n})\cdot\overline{\vect{E}}\dd S=0.
\end{equation*}
The far-field asymptotics of the volume potential give
\begin{equation*}
\operatorname{Im}\int_{\partial B_\rho}(\curl\vect{E}\times\vect{n})\cdot\overline{\vect{E}}\dd S
=k\int_{\partial B_\rho}|\vect{E}|^2\dd S+O(\rho^{-1}),\quad \rho\to\infty.
\end{equation*}
Passing to the limit yields
\begin{equation*}
k^2\int_D\operatorname{Im}\!\left[\vect{E}^*(\Id+\Qt)\vect{E}\right]\dd\vect{x}
+k\lim_{\rho\to\infty}\int_{\partial B_\rho}|\vect{E}|^2\dd S=0.
\end{equation*}
Strict dissipativity implies $\vect{u}=\vect{E}|_D = 0$. The Fredholm alternative therefore gives a unique solution of the volume equation and continuous dependence on the incident field. Formula~\eqref{eq:volume_integral_equation} extends it to the unique radiating Maxwell solution in $H_{\mathrm{loc}}(\operatorname{curl};\R^3)$.
\end{proof}

The Green-tensor convolution for interior points is understood in the distributional sense.
We now use the volume representation to obtain the far-field pattern. For bounded $\vect{y}$, the Helmholtz fundamental solution satisfies
\begin{equation*}
\Phi_k(\vect{x}-\vect{y})=\frac{e^{\ii k|\vect{x}|}}{4\pi |\vect{x}|}e^{-\ii k\hat{\vect{x}}\cdot\vect{y}}+O(|\vect{x}|^{-2}),\quad |\vect{x}|\to\infty,
\end{equation*}
uniformly with respect to $\hat{\vect{x}}\in\Sph^2$. It follows that the Maxwell Green tensor has the asymptotic expansion
\begin{equation}\label{eq:green_farfield}
\G_k(\vect{x}-\vect{y})=\frac{e^{\ii k|\vect{x}|}}{4\pi |\vect{x}|}\tens{P}_{\hat{\vect{x}}}e^{-\ii k\hat{\vect{x}}\cdot\vect{y}}+O(|\vect{x}|^{-2}),\quad \tens{P}_{\hat{\vect{x}}}=\Id-\hat{\vect{x}}\hat{\vect{x}}^\top.
\end{equation}
The matrix $\tens{P}_{\hat{\vect{x}}}$ is the orthogonal projection onto the plane perpendicular to $\hat{\vect{x}}$.

Substituting \eqref{eq:green_farfield} into \eqref{eq:volume_representation} gives the following far-field representation.

\begin{corollary}\label{cor:farfield}
Let $\vect{E}$ be the radiating solution considered in Theorem~\ref{prop:volume_integral}. Then the scattered field satisfies
\begin{equation}\label{eq:far_field_expansion}
\vect{E}^{\mathrm s}(\vect{x};\vect{d},\vect{e};k)=\frac{e^{\ii k|\vect{x}|}}{|\vect{x}|}\vect{E}^{\infty}(\hat{\vect{x}};\vect{d},\vect{e};k)+O(|\vect{x}|^{-2}),\quad |\vect{x}|\to\infty,
\end{equation}
where
\begin{equation}\label{eq:exact_farfield}
\vect{E}^{\infty}(\hat{\vect{x}};\vect{d},\vect{e};k)=\frac{k^2}{4\pi}\tens{P}_{\hat{\vect{x}}}\int_D e^{-\ii k\hat{\vect{x}}\cdot\vect{y}}\Qt(\vect{y})\vect{E}(\vect{y};\vect{d},\vect{e};k)\dd\vect{y}.
\end{equation}
In particular,
\begin{equation*}
\hat{\vect{x}}\cdot\vect{E}^{\infty}(\hat{\vect{x}};\vect{d},\vect{e};k)=0,
\end{equation*}
so the electric far-field pattern is transverse to the observation direction.
\end{corollary}

\begin{proof}
Inserting \eqref{eq:green_farfield} into \eqref{eq:volume_representation} yields
\begin{equation*}
\vect{E}^{\mathrm s}(\vect{x})=\frac{e^{\ii k|\vect{x}|}}{|\vect{x}|}\frac{k^2}{4\pi}\tens{P}_{\hat{\vect{x}}}\int_D e^{-\ii k\hat{\vect{x}}\cdot\vect{y}}\Qt(\vect{y})\vect{E}(\vect{y})\dd\vect{y}+O(|\vect{x}|^{-2}).
\end{equation*}
Comparison with \eqref{eq:far_field_expansion} gives \eqref{eq:exact_farfield}. Since $\tens{P}_{\hat{\vect{x}}}\hat{\vect{x}}=\vect{0}$ and $\tens{P}_{\hat{\vect{x}}}$ is symmetric, the far-field pattern satisfies $\hat{\vect{x}}\cdot\vect{E}^{\infty}=0$.
\end{proof}

To obtain a direct reconstruction, we replace the internal total field in \eqref{eq:exact_farfield} by the incident plane wave. Define the entrywise Fourier transform by
\begin{equation*}\label{eq:fourier_tensor_definition}
 \widehat{\Qt}(\vect{\xi})=\int_D e^{\ii\vect{\xi}\cdot\vect{y}}\Qt(\vect{y})\dd\vect{y}.
\end{equation*}
The Born far field can then be expressed directly in terms of $\widehat{\Qt}$.
\begin{theorem}\label{thm:born_tensor}
For the incident plane wave $\vect{E}^{\mathrm i}(\vect{x})=\vect{e}\,e^{\ii k\vect{d}\cdot\vect{x}}$, where $\vect{d}\in\Sph^2$ and $\vect{e}\cdot\vect{d}=0$, the Born approximation of the electric far-field pattern is
\begin{equation}\label{eq:born_tensor}
\vect{E}_b^\infty(\hat{\vect{x}};\vect{d},\vect{e};k)=\frac{k^2}{4\pi}\tens{P}_{\hat{\vect{x}}}\widehat{\Qt}\bigl(k(\vect{d}-\hat{\vect{x}})\bigr)\vect{e},\quad \hat{\vect{x}}\in\Sph^2.
\end{equation}
\end{theorem}
\begin{proof}
Substituting $\vect{E}\approx \vect{E}^{\rm i}$ into \eqref{eq:exact_farfield} gives
\begin{align*}
\vect{E}_b^\infty(\hat{\vect{x}};\vect{d},\vect{e};k)&=\frac{k^2}{4\pi}\tens{P}_{\hat{\vect{x}}}\int_D e^{-\ii k\hat{\vect{x}}\cdot\vect{y}}\Qt(\vect{y})\vect{e}\,e^{\ii k\vect{d}\cdot\vect{y}}\dd\vect{y}\\
&=\frac{k^2}{4\pi}\tens{P}_{\hat{\vect{x}}}\left[\int_D e^{\ii k(\vect{d}-\hat{\vect{x}})\cdot\vect{y}}\Qt(\vect{y})\dd\vect{y}\right]\vect{e}\\
&=\frac{k^2}{4\pi}\tens{P}_{\hat{\vect{x}}}\widehat{\Qt}\bigl(k(\vect{d}-\hat{\vect{x}})\bigr)\vect{e}.
\end{align*}
\end{proof}

The Born relation is obtained by replacing the internal total field by the incident field.  Because the numerical experiments below use data from the full Maxwell volume integral equation, we record here the resulting modeling remainder.  Substituting $\vect{E}=\vect{E}^{\mathrm i}+\vect{E}^{\mathrm s}$ into the exact far-field representation \eqref{eq:exact_farfield} gives
\begin{equation*}\label{eq:exact_born_decomp}
\vect{E}^{\infty}=\vect{E}_b^{\infty}+\mathcal E_{\mathrm{nl}}^{\infty},
\end{equation*}
where
\begin{equation}\label{eq:nonlinear_farfield}
\mathcal E_{\mathrm{nl}}^{\infty}(\hat{\vect{x}};\vect{d},\vect{e};k)=\frac{k^2}{4\pi}\tens{P}_{\hat{\vect{x}}}\int_D e^{-\ii k\hat{\vect{x}}\cdot\vect{y}}\Qt(\vect{y})\vect{E}^{\mathrm s}(\vect{y};\vect{d},\vect{e};k)\dd\vect{y}.
\end{equation}
The next proposition gives an a priori bound under a small-norm condition on the volume integral operator.

\begin{proposition}\label{prop:born_error}
Let $\Qt$ satisfy Assumption~\ref{ass:material} and suppose
\begin{equation*}
\rho_B=k^2\|V_Q\|_{L^2(D)^3\to L^2(D)^3}<1.
\end{equation*}
Then
\begin{equation*}
\|\vect{E}^{\mathrm s}\|_{L^2(D)^3}\leq\frac{\rho_B}{1-\rho_B}\|\vect{E}^{\mathrm i}\|_{L^2(D)^3},
\end{equation*}
and
\begin{equation}\label{eq:born_remainder_bound}
\|\mathcal E_{\mathrm{nl}}^{\infty}\|_{L^2(\Sph^2)^3}
\leq\frac{k^2|D|^{1/2}}{\sqrt{4\pi}}\|\Qt\|_{L^\infty(D)}
\frac{\rho_B}{1-\rho_B}\|\vect{E}^{\mathrm i}\|_{L^2(D)^3},
\end{equation}
where the contrast norm uses the induced matrix norm.
\end{proposition}

\begin{proof}
The volume equation and the Neumann series give
\begin{equation*}
\vect{E}=(I-k^2V_Q)^{-1}\vect{E}^{\mathrm i},
\qquad
\|(I-k^2V_Q)^{-1}\|\leq\frac{1}{1-\rho_B}.
\end{equation*}
Since $\vect{E}^{\mathrm s}=k^2V_Q\vect{E}$ in $D$,
\begin{equation*}
\|\vect{E}^{\mathrm s}\|_{L^2(D)^3}
\leq\rho_B\|\vect{E}\|_{L^2(D)^3}
\leq\frac{\rho_B}{1-\rho_B}\|\vect{E}^{\mathrm i}\|_{L^2(D)^3}.
\end{equation*}
Using $\|\tens{P}_{\hat{\vect{x}}}\|=1$ and Cauchy--Schwarz in \eqref{eq:nonlinear_farfield},
\begin{equation*}
|\mathcal E_{\mathrm{nl}}^{\infty}(\hat{\vect{x}};\vect{d},\vect{e};k)|
\leq\frac{k^2}{4\pi}|D|^{1/2}\|\Qt\|_{L^\infty(D)}\|\vect{E}^{\mathrm s}\|_{L^2(D)^3}.
\end{equation*}
Integration over $\Sph^2$ and substitution of the scattered-field estimate yield \eqref{eq:born_remainder_bound}.
\end{proof}

Proposition~\ref{prop:born_error} isolates the discrepancy between the full Maxwell far field and the Born data used by the reconstruction.  This remainder is not inverted independently: in the discrete reconstruction it enters the stacked polarimetric data, is amplified by the local pseudoinverse, and is then propagated through the spectral cutoff.  Section~\ref{sec:discretization} incorporates these effects in the combined reconstruction error estimate.

Formula~\eqref{eq:born_tensor} relates the linearized Maxwell far field to the Fourier transform of the anisotropic contrast. However, the data do not provide $\widehat{\Qt}(\vect{\xi})$ directly. For each incidence--observation pair, they contain only its action on the admissible polarization $\vect{e}$, followed by the transverse projection $\tens{P}_{\hat{\vect{x}}}$. Recovering the tensor Fourier value at a fixed point $\vect{\xi}$ therefore requires several measurements associated with the same Fourier argument. The next section develops this polarimetric recovery.

\section{Polarimetric identifiability and conditioning}\label{sec:born}

Under the Born relation \eqref{eq:born_tensor}, the data at a Fourier point $\vect\xi$ are transverse input-output actions of $\widehat{\Qt}(\vect\xi)$.  We first characterize which Fourier points are accessible and then identify a minimal acquisition for a general tensor.

\begin{lemma}\label{lem:fourier_ball_coverage}
The set $\{k(\vect d-\hat{\vect x}):\vect d,\hat{\vect x}\in\Sph^2\}$ equals $\overline{B(0,2k)}$.  For $\vect\xi\in\overline{B(0,2k)}$, set $\vect p=\vect\xi/(2k)$ and choose $\vect s\in\Sph^2\cap\vect p^\perp$.  Then
\begin{equation}\label{eq:direction_construction}
 \vect d=\vect p+\sqrt{1-|\vect p|^2}\,\vect s,
 \qquad
 \hat{\vect x}=-\vect p+\sqrt{1-|\vect p|^2}\,\vect s
\end{equation}
belong to $\Sph^2$ and satisfy $k(\vect d-\hat{\vect x})=\vect\xi$.
\end{lemma}
\begin{proof}
Orthogonality of $\vect p$ and $\vect s$ gives unit length, and subtraction in \eqref{eq:direction_construction} gives $2k\vect p=\vect\xi$.  Conversely, $|k(\vect d-\hat{\vect x})|\le 2k$.
\end{proof}

The freedom in choosing $\vect{s}$ allows us to combine different polarimetric measurements of the same tensor Fourier value, leading to the linear system below.

\subsection{Polarimetric linear system}
Let $\mathcal T\subset\C^{3\times3}$ be a tensor class of dimension $m$ with Frobenius-orthonormal basis $\{\Tt_r\}_{r=1}^m$.  At a fixed Fourier point write
\begin{equation}\label{eq:tensor_basis}
 \widehat{\Qt}(\vect\xi)=\sum_{r=1}^m c_r(\vect\xi)\Tt_r.
\end{equation}
For $J$ pairs $(\vect d_j,\hat{\vect x}_j)$ satisfying $k(\vect d_j-\hat{\vect x}_j)=\vect\xi$, let $\tens E_j=(\vect e_{j,1},\vect e_{j,2})$ be an orthonormal basis of $\vect d_j^\perp$.  Define the normalized $3\times2$ data block
\begin{equation}\label{eq:born_data_block}
 \Bt_j=\frac{4\pi}{k^2}
 \bigl(\vect E_b^\infty(\hat{\vect x}_j;\vect d_j,\vect e_{j,1};k),
       \vect E_b^\infty(\hat{\vect x}_j;\vect d_j,\vect e_{j,2};k)\bigr).
\end{equation}
Then
\begin{equation*}\label{eq:polarized_data_block}
 \Bt_j=\tens P_{\hat{\vect x}_j}\widehat{\Qt}(\vect\xi)\tens E_j.
\end{equation*}
With columnwise vectorization, define 
\begin{equation}\label{equ:defineofgj}
g_j=\vecop(\Bt_j)
\end{equation}
and
\begin{equation}\label{eq:measurement_block}
 [\Mt_j]_{:,r}=\vecop\!\left(\tens P_{\hat{\vect x}_j}\Tt_r\tens E_j\right).
\end{equation}
Stacking the blocks gives
\begin{equation}\label{eq:polarimetric_system}
 g=\Mt c,\qquad
 \Mt=\begin{pmatrix}\Mt_1\\ \vdots\\ \Mt_J\end{pmatrix}\in\C^{6J\times m}.
\end{equation}
For the general class $\mathcal T=\C^{3\times3}$, $m=9$; for symmetric tensors $m=6$; and for isotropic tensors $m=1$.

The following proposition gives the polarimetric reconstruction and bounds its error in terms of the data perturbation and the smallest singular value of the measurement matrix.

\begin{proposition}\label{prop:tensor_recovery}
If $\rank\Mt=m$, then $c=\Mt^\dagger g$. If $g^\delta=g+\eta$, the reconstruction
$\widehat{\Qt}^{\,\delta}(\vect{\xi})=\sum_{r=1}^{m}[\Mt^\dagger g^\delta]_r\Tt_r$
satisfies
\begin{equation*}\label{eq:polarimetric_stability}
 \|\widehat{\Qt}^{\,\delta}(\vect\xi)-\widehat{\Qt}(\vect\xi)\|_{\Frob}
 \le \frac{\|\eta\|_2}{\sigma_{\min}(\Mt)}.
\end{equation*}
\end{proposition}
\begin{proof}
Because the columns represent a Frobenius-orthonormal tensor basis, the coefficient Euclidean norm equals the tensor Frobenius norm.  Full column rank gives $\Mt^\dagger\Mt=I_m$ and $\|\Mt^\dagger\|_2=\sigma_{\min}(\Mt)^{-1}$.
\end{proof}

\subsection{A minimal three-configuration design}
We now use the construction in Lemma~\ref{lem:fourier_ball_coverage} to select a minimal set of measurement configurations satisfying the full-rank condition in Proposition~\ref{prop:tensor_recovery}.

\begin{theorem}\label{thm:three_configuration}
Let $0<|\vect{\xi}|<2k$, $\rho=|\vect{\xi}|/(2k)$, and $\vect{n}=\vect{\xi}/|\vect{\xi}|$. Choose an orthonormal basis $\vect{u},\vect{v}$ of $\vect{n}^{\perp}$ and set
\begin{equation*}\label{eq:three_angles}
\theta_j=\frac{2\pi(j-1)}{3},\qquad \vect{s}_j=\cos\theta_j\,\vect{u}+\sin\theta_j\,\vect{v},\qquad j=1,2,3.
\end{equation*}
For the configurations
\begin{equation*}\label{eq:three_directions}
\vect{d}_j=\rho\vect{n}+\sqrt{1-\rho^2}\,\vect{s}_j,\qquad \hat{\vect{x}}_j=-\rho\vect{n}+\sqrt{1-\rho^2}\,\vect{s}_j,
\end{equation*}
each using two orthogonal incident polarizations, the polarimetric matrix for $\mathcal T=\C^{3\times3}$ satisfies
\begin{equation*}\label{eq:gram_determinant}
\det(\Mt^*\Mt)=\frac{3^9}{2^9}\rho^2(1-\rho^2)^6(9\rho^4+2\rho^2+1)^2>0,\qquad \rank\Mt=9.
\end{equation*}
Three configurations are minimal: any choice with $J\leq2$ satisfies $\rank\Mt\leq4J<9$.
\end{theorem}

\begin{proof}
Each configuration gives $\rank\Mt_j\leq4$, so $\rank\Mt\leq4J<9$ when $J\leq2$.

For sufficiency, choose the standard Frobenius tensor basis. Columnwise vectorization gives $\Mt_j=\tens{E}_j^\top\otimes\tens{P}_{\hat{\vect{x}}_j}$ and hence
\begin{equation*}
\Mt^*\Mt=\sum_{j=1}^{3}\tens{P}_{\vect{d}_j}\otimes\tens{P}_{\hat{\vect{x}}_j}.
\end{equation*}
In coordinates aligned with $\vect{u},\vect{v},\vect{n}$, the calculation in Appendix~\ref{app:three_configuration} gives the stated determinant, which is positive for $0<\rho<1$. Thus $\rank\Mt=9$. Any other Frobenius-orthonormal tensor basis gives a unitarily equivalent Gram matrix and preserves its determinant.
\end{proof}

We next examine the conditioning of this three-configuration design near the center and boundary of the Fourier ball.

\begin{proposition}\label{prop:center_boundary_conditioning}
For the three-configuration design in Theorem~\ref{thm:three_configuration},
\begin{equation*}
\sigma_{\min}(\Mt(\rho))\sim\sqrt{3}\,\rho
\quad\text{as }\rho\to0^+,
\quad
\sigma_{\min}(\Mt(\rho))\sim\sqrt{\frac32}\,(1-\rho^2)
\quad\text{as }\rho\to1^-.
\end{equation*}
The degeneracy at $\vect{\xi}=\vect{0}$ can be avoided by choosing $\vect{d}_j=\hat{\vect{x}}_j=\vect{a}_j$, where $\{\vect{a}_1,\vect{a}_2,\vect{a}_3\}$ is an orthonormal basis of $\R^3$. With both orthogonal incident polarizations in each configuration, $\rank\Mt=9$ and $\det(\Mt^*\Mt)=8$.
At $|\vect{\xi}|=2k$, however, every admissible configuration satisfies $\vect{d}=\vect{\xi}/|\vect{\xi}|$ and $\hat{\vect{x}}=-\vect{d}$, so $\rank\Mt\leq4$ regardless of the number of configurations. Thus the boundary degeneracy cannot be removed by changing the measurement configurations.
\end{proposition}

\begin{proof}
The center and boundary rank statements follow directly from the corresponding configurations. It remains to establish the two asymptotic formulas. The block decomposition in \eqref{eq:appendix_gram_blocks} gives
\begin{align*}
\det(\lambda\Id-\Mt(\rho)^*\Mt(\rho))
&=\frac{1}{512}(\lambda-3\rho^2)(2\lambda+3\rho^2-3)^2\\
&\quad\times\left[2\lambda^2-(9\rho^4-12\rho^2+9)\lambda+9(1-\rho^2)^2\right]\\
&\quad\times\left[8\lambda^2+(18\rho^4-24\rho^2-18)\lambda
+9(1-\rho^2)(9\rho^4+2\rho^2+1)\right]^2.
\end{align*}
At $\rho=0$, only the root $3\rho^2$ tends to zero; all other roots have positive limits. Therefore
\begin{equation*}
\lambda_{\min}(\Mt(\rho)^*\Mt(\rho))\sim 3\rho^2
\quad\text{as }\rho\to0^+.
\end{equation*}
As $\rho\to1^-$, the first quadratic factor has a root tending to $3$, so its smaller root is asymptotic to $\frac32(1-\rho^2)^2$. The small roots of the second quadratic factor are asymptotic to $\frac92(1-\rho^2)$, while the remaining roots are $3\rho^2$ and $\frac32(1-\rho^2)$. Consequently,
\begin{equation*}
\lambda_{\min}(\Mt(\rho)^*\Mt(\rho))\sim\frac32(1-\rho^2)^2
\quad\text{as }\rho\to1^-.
\end{equation*}
Taking square roots yields the stated asymptotics for $\sigma_{\min}(\Mt(\rho))$.
\end{proof}

The three-configuration design guarantees unique recovery at nonzero interior Fourier points, but small singular values may amplify data errors near the center and boundary. The center degeneracy can be avoided by the three-axis configuration, whereas the boundary degeneracy is unavoidable. We therefore use the smallest singular value to guide the selection of measurement configurations in the discrete reconstruction in Section~\ref{sec:discretization}.

\section{Low-rank inversion on the Fourier ball}\label{sec:gpswf}

After recovering the tensor Fourier values, we reconstruct the contrast by inverting a restricted Fourier operator. Let $\supp\Qt\subset B(0,R)$, write $B_1=B(0,1)$, and set $\Qt_R(\vect{z})=\Qt(R\vect{z})$. With $\vect{\xi}=2k\vect{p}$ and $C=2kR$, a change of variables gives
\begin{equation}\label{eq:scaling}
R^{-3}\widehat{\Qt}(2k\vect{p})=\int_{B_1}e^{\ii C\vect{p}\cdot\vect{z}}\Qt_R(\vect{z})\dd\vect{z},\qquad \vect{p}\in B_1.
\end{equation}
Define
\begin{equation}\label{eq:restricted_fourier}
(\F_C f)(\vect{p})=\int_{B_1}e^{\ii C\vect{p}\cdot\vect{z}}f(\vect{z})\dd\vect{z},\qquad \vect{p}\in B_1.
\end{equation}
The operator acts componentwise on matrix-valued functions, so \eqref{eq:scaling} reads $R^{-3}\widehat{\Qt}(2k\,\cdot)=\F_C\Qt_R$. Since $\F_C$ is compact, its inversion requires regularization based on its spectral structure.

The eigenfunctions of $\F_C$ are the three-dimensional ball generalized prolate spheroidal wave functions (GPSWFs) \cite{Greengard2024,Slepian1964,ZhangLiWangZhang2020}. They form an orthonormal basis of $L^2(B_1)$ and satisfy
\begin{equation}\label{eq:gpswf_eigen}
\F_C\psi_{\ell n m}=\alpha_{\ell n}(C)\psi_{\ell n m}.
\end{equation}
In spherical coordinates, these functions have the separated form
\begin{equation}\label{eq:gpswf_separation}
\psi_{\ell n m}(r,\vect{\omega};C)=R_{\ell n}(r;C)Y_\ell^m(\vect{\omega}),
\end{equation}
where $\ell,n\geq0$, $-\ell\leq m\leq\ell$, and $Y_\ell^m$ are orthonormal spherical harmonics on $\Sph^2$. The radial functions are normalized in $L^2((0,1),r^2\dd r)$, and the eigenvalues $\alpha_{\ell n}(C)$ are independent of $m$. Their numerical construction is described in Appendix~\ref{sec:gpswf_implementation}.

We now expand the contrast and its recovered Fourier data in this basis. Writing $W(\vect{p})=R^{-3}\widehat{\Qt}(2k\vect{p})$, we have
\begin{equation}\label{eq:modal_expansion}
\Qt_R=\sum_\nu\tens{Q}_\nu\psi_\nu,\qquad W=\sum_\nu W_\nu\psi_\nu,
\end{equation}
where $\nu$ represents $(\ell,n,m)$ and the coefficients $\tens{Q}_\nu$ and $W_\nu$ are $3\times3$ matrices. Since $W=\F_C\Qt_R$, equation~\eqref{eq:gpswf_eigen} gives $W_\nu=\alpha_\nu(C)\tens{Q}_\nu$. Thus the spatial coefficients are recovered by
\begin{equation}\label{eq:modal_inverse}
\tens{Q}_\nu=\frac{W_\nu}{\alpha_\nu(C)},\qquad \alpha_\nu(C)\neq0.
\end{equation}

Division by small eigenvalues in \eqref{eq:modal_inverse} amplifies errors in the recovered Fourier data. We therefore retain only the modes in
\begin{equation*}
\mathcal J_\epsilon=\{\nu:|\alpha_\nu(C)|>\epsilon\},\qquad \epsilon>0.
\end{equation*}
For perturbed Fourier data $W^\delta$, define
\begin{equation}\label{eq:Q_cutoff}
\Qt_{R,\epsilon}^\delta=\sum_{\nu\in\mathcal J_\epsilon}\frac{W_\nu^\delta}{\alpha_\nu(C)}\psi_\nu,\qquad W_\nu^\delta=\int_{B_1}W^\delta(\vect{p})\overline{\psi_\nu(\vect{p})}\dd\vect{p}.
\end{equation}
The reconstruction in the original spatial coordinates is then $\Qt_\epsilon^\delta(\vect{x})=\Qt_{R,\epsilon}^\delta(\vect{x}/R).$ The following estimate separates the effect of errors in the recovered Fourier data from the error caused by discarding the unstable modes.

\begin{theorem}\label{thm:spectral_cutoff}
For the expansions in \eqref{eq:modal_expansion}, suppose that
\begin{equation*}
\|W^\delta-W\|_{L^2(B_1;\C^{3\times3})}\leq\delta_F.
\end{equation*}
Then the reconstruction \eqref{eq:Q_cutoff} satisfies
\begin{equation}\label{eq:cutoff_error}
\|\Qt_{R,\epsilon}^\delta-\Qt_R\|_{L^2(B_1;\C^{3\times3})}
\leq\frac{\delta_F}{\epsilon}
+\left(\sum_{\nu\notin\mathcal J_\epsilon}\|\tens{Q}_\nu\|_{\Frob}^2\right)^{1/2}.
\end{equation}
\end{theorem}

\begin{proof}
The reconstruction error consists of the data error and the truncation error. By the orthonormality of the GPSWF basis and $|\alpha_\nu(C)|>\epsilon$ for $\nu\in\mathcal J_\epsilon$,
\begin{equation*}
\left\|\sum_{\nu\in\mathcal J_\epsilon}\frac{W_\nu^\delta-W_\nu}{\alpha_\nu(C)}\psi_\nu\right\|_{L^2}\leq\frac{\delta_F}{\epsilon}.
\end{equation*}
The discarded modes have norm
\begin{equation*}
\left\|\sum_{\nu\notin\mathcal J_\epsilon}\tens{Q}_\nu\psi_\nu\right\|_{L^2}=\left(\sum_{\nu\notin\mathcal J_\epsilon}\|\tens{Q}_\nu\|_{\Frob}^2\right)^{1/2}.
\end{equation*}
The result follows from the triangle inequality.
\end{proof}

The first term in \eqref{eq:cutoff_error} describes the amplification of errors in the recovered Fourier data, while the second term is the truncation error. Decreasing $\epsilon$ retains more spatial modes but also increases the amplification of data errors. The cutoff parameter therefore balances these two effects. Here $\delta_F$ measures the error after the polarimetric recovery and therefore includes the effect of the conditioning of $\Mt(\vect{\xi})$.

\section{Discrete Acquisition and Reconstruction Stability}\label{sec:discretization}

This section presents the discrete reconstruction from a finite set of incident and observation directions. We describe the algorithm and then analyze the errors introduced by the discrete acquisition and reconstruction.

\subsection{Discrete reconstruction}

We first implement the polarimetric recovery and GPSWF inversion described above and summarize the procedure in an algorithm. Let $\mathcal D=\{\vect{d}_a\}_{a=1}^{N_d}\subset\Sph^2$ and $\mathcal X=\{\hat{\vect{x}}_b\}_{b=1}^{N_x}\subset\Sph^2$ denote the available directions. These generate the normalized Fourier points
\begin{equation*}
\mathcal P=\left\{\frac{\vect{d}_a-\hat{\vect{x}}_b}{2}:1\leq a\leq N_d,\ 1\leq b\leq N_x\right\}\subset\overline{B_1}.
\end{equation*}

Let $\{(\vect{p}_\mu,w_\mu)\}_{\mu=1}^{N_q}$ be the quadrature nodes and weights used in the GPSWF projection. These nodes generally do not belong to $\mathcal P$. For each $\vect{p}_\mu$, we select $J$ nearby pairs from $\mathcal D\times\mathcal X$, denoted by $(\vect{d}_{\mu,j},\hat{\vect{x}}_{\mu,j})$, such that
\begin{equation*}
\vect{p}_{\mu,j}=\frac{\vect{d}_{\mu,j}-\hat{\vect{x}}_{\mu,j}}{2}\approx\vect{p}_\mu,\qquad j=1,\ldots,J.
\end{equation*}
For each selected pair, we choose two orthonormal incident polarizations spanning $\vect{d}_{\mu,j}^{\perp}$. The configurations are selected to give a full-column-rank polarimetric matrix, favoring a larger smallest singular value among configurations with comparable Fourier-point distances.

For each $\vect{p}_\mu$, let $g_\mu^\delta$ be formed as in \eqref{equ:defineofgj}, using the measured far fields in place of the Born far fields. Assemble $\Mt_\mu$ from the selected configurations according to \eqref{eq:measurement_block}. Applying the pseudoinverse formula in Proposition~\ref{prop:tensor_recovery} gives
\begin{equation}\label{eq:discrete_fourier_sample}
\widetilde W_\mu=R^{-3}\sum_{r=1}^{m}\bigl[\Mt_\mu^\dagger g_\mu^\delta\bigr]_r\Tt_r.
\end{equation}

Using these samples, we approximate the Fourier-data coefficients in \eqref{eq:Q_cutoff} by quadrature:
\begin{equation*}
\widetilde W_\nu=\sum_{\mu=1}^{N_q}w_\mu\widetilde W_\mu\overline{\psi_\nu(\vect{p}_\mu)}.
\end{equation*}
The discrete reconstruction is then
\begin{equation}\label{eq:discrete_reconstruction}
\widetilde{\Qt}(\vect{x})=\sum_{\nu\in\mathcal J_\epsilon}\frac{\widetilde W_\nu}{\alpha_\nu(C)}\psi_\nu(\vect{x}/R),\qquad \vect{x}\in B(0,R),
\end{equation}
and is set to zero outside $B(0,R)$.

The reconstruction is summarized in Algorithm~\ref{alg:discrete_reconstruction}.

\begin{algorithm}[ht]
\caption{Discrete polarimetric GPSWF reconstruction}
\label{alg:discrete_reconstruction}
\begin{algorithmic}[1]
\Require Measured far-field data, direction sets $\mathcal D,\mathcal X$, quadrature rule $\{(\vect{p}_\mu,w_\mu)\}_{\mu=1}^{N_q}$, tensor basis $\{\Tt_r\}_{r=1}^{m}$, support radius $R$, and retained eigenpairs $\{(\psi_\nu,\alpha_\nu(C))\}_{\nu\in\mathcal J_\epsilon}$
\For{$\mu=1,\ldots,N_q$}
\State Select $J$ nearby configurations giving a full-column-rank matrix $\Mt_\mu$
\State Assemble $g_\mu^\delta$ and $\Mt_\mu$
\State Compute $\widetilde W_\mu=R^{-3}\sum_{r=1}^{m}[\Mt_\mu^\dagger g_\mu^\delta]_r\Tt_r$
\EndFor
\For{$\nu\in\mathcal J_\epsilon$}
\State Compute $\widetilde W_\nu=\sum_{\mu=1}^{N_q}w_\mu\widetilde W_\mu\overline{\psi_\nu(\vect{p}_\mu)}$
\EndFor
\State \Return $\widetilde{\Qt}$ given by \eqref{eq:discrete_reconstruction}
\end{algorithmic}
\end{algorithm}

In the numerical experiments, the quadrature on $B_1$ combines a radial Gauss--Jacobi rule with a Lebedev rule on $\Sph^2$. Let $(\eta_a,w_a^{(r)})$ be the Gauss--Jacobi nodes and weights for $(1+\eta)^{1/2}$ on $[-1,1]$, and let $(\vect{\omega}_b,w_b^{(s)})$ be the Lebedev nodes and weights. The resulting ball quadrature is
\begin{equation*}
\vect{p}_{ab}=\sqrt{\frac{1+\eta_a}{2}}\,\vect{\omega}_b,\qquad
w_{ab}=\frac{1}{4\sqrt2}w_a^{(r)}w_b^{(s)}.
\end{equation*}
The pair $(a,b)$ is represented by the single index $\mu$ above.

\subsection{Error analysis}

We now estimate the errors in the discrete reconstruction. The selected Fourier points generally differ from the target quadrature nodes. For each $\vect{p}_\mu$, let $h_\mu=\max_{1\leq j\leq J}|\vect{p}_{\mu,j}-\vect{p}_\mu|$. Since $\supp\Qt\subset B(0,R)$, the Fourier transform satisfies
\begin{align}
\|\widehat{\Qt}(2k\vect{p}')-\widehat{\Qt}(2k\vect{p})\|_{\Frob}
&\leq\int_D|e^{2\ii k\vect{p}'\cdot\vect{x}}-e^{2\ii k\vect{p}\cdot\vect{x}}|\,\|\Qt(\vect{x})\|_{\Frob}\dd\vect{x} \nonumber \\
&\leq2kR|\vect{p}'-\vect{p}|\,\|\Qt\|_{L^1(D;\Frob)}. \label{equ:hatqmismatch}
\end{align}
Here $\|\Qt\|_{L^1(D;\Frob)}=\int_D\|\Qt(\vect{x})\|_{\Frob}\dd\vect{x}$. Thus the off-grid error is at most linear in the distance between the measurement and target nodes.

The measured data are obtained by adding noise to the full Maxwell far fields. After normalization and stacking, the resulting vector satisfies
\begin{equation*}\label{eq:discrete_data_decomposition}
g_\mu^\delta=\Mt_\mu c(2k\vect{p}_\mu)+r_\mu+b_\mu+\eta_\mu,
\end{equation*}
where $c$ is defined by \eqref{eq:tensor_basis}, $\eta_\mu$ is the normalized measurement noise, and $b_\mu$ is the normalized and stacked Born remainder \eqref{eq:nonlinear_farfield}. The off-grid contribution is
\begin{equation*}
r_\mu=
\begin{pmatrix}
\Mt_{\mu,1}\bigl[c(2k\vect{p}_{\mu,1})-c(2k\vect{p}_\mu)\bigr]\\
\vdots\\
\Mt_{\mu,J}\bigl[c(2k\vect{p}_{\mu,J})-c(2k\vect{p}_\mu)\bigr]
\end{pmatrix}.
\end{equation*}
The transverse projections and orthonormal polarizations do not increase the Frobenius norm. The preceding Fourier estimate (\ref{equ:hatqmismatch}) therefore gives
\begin{equation*}
\|r_\mu\|_2\leq2kR\sqrt{J}\,h_\mu\|\Qt\|_{L^1(D;\Frob)}.
\end{equation*}

Write $W_\mu=W(\vect{p}_\mu)=R^{-3}\widehat{\Qt}(2k\vect{p}_\mu)$ for the exact scaled Fourier value, and let $\widetilde W_\mu$ be its approximation defined in \eqref{eq:discrete_fourier_sample}. Proposition~\ref{prop:tensor_recovery}, applied to the total perturbation $r_\mu+b_\mu+\eta_\mu$, gives the following estimate.

\begin{proposition}\label{prop:offgrid}
If $\Mt_\mu$ has full column rank, the recovered sample \eqref{eq:discrete_fourier_sample} satisfies
\begin{equation*}\label{eq:offgrid_bound}
\|\widetilde W_\mu-W_\mu\|_{\Frob}
\leq
\frac{R^{-3}}{\sigma_{\min}(\Mt_\mu)}
\left(\|\eta_\mu\|_2+\|b_\mu\|_2
+2kR\sqrt{J}\,h_\mu\|\Qt\|_{L^1(D;\Frob)}\right).
\end{equation*}
\end{proposition}

Assuming positive quadrature weights, we collect the off-grid and data-perturbation contributions separately:
\begin{equation}\label{eq:discrete_error_components}
\begin{aligned}
\Delta_{\mathrm{off}}^2
&=4k^2R^{-4}J\|\Qt\|_{L^1(D;\Frob)}^2
\sum_{\mu=1}^{N_q}
\frac{w_\mu h_\mu^2}{\sigma_{\min}(\Mt_\mu)^2},\\
\Delta_{\mathrm{data}}^2
&=R^{-6}\sum_{\mu=1}^{N_q}
\frac{w_\mu(\|\eta_\mu\|_2+\|b_\mu\|_2)^2}
{\sigma_{\min}(\Mt_\mu)^2}.
\end{aligned}
\end{equation}
Proposition~\ref{prop:offgrid} and the weighted triangle inequality give
\begin{equation*}\label{eq:combined_node_error_bound}
\left(\sum_{\mu=1}^{N_q}w_\mu
\|\widetilde W_\mu-W_\mu\|_{\Frob}^2\right)^{1/2}
\leq\Delta_{\mathrm{off}}+\Delta_{\mathrm{data}}.
\end{equation*}

The quadrature introduces a further error even when the exact samples $W_\mu$ are used. For each retained mode, define
\begin{equation*}\label{eq:quadrature_error_def}
q_\nu=\sum_{\mu=1}^{N_q}w_\mu W_\mu\overline{\psi_\nu(\vect{p}_\mu)}
-\int_{B_1}W(\vect{p})\overline{\psi_\nu(\vect{p})}\dd\vect{p},
\end{equation*}
and set \begin{equation}\label{eq:quadrature_error_norm}
\Delta_{\mathrm{quad}}^2
=\sum_{\nu\in\mathcal J_\epsilon}\|q_\nu\|_{\Frob}^2.
\end{equation}

The propagation of sample errors to the retained coefficients is described by the weighted sampling matrix
\begin{equation}\label{eq:sampling_matrix_A}
A_{\nu\mu}=\sqrt{w_\mu}\,\overline{\psi_\nu(\vect{p}_\mu)},
\qquad \nu\in\mathcal J_\epsilon,\quad \mu=1,\ldots,N_q.
\end{equation}
Writing $C_h=\|A\|_2$, we obtain
\begin{align*}
\left(\sum_{\nu\in\mathcal J_\epsilon}\|\widetilde W_\nu-W_\nu\|_{\Frob}^2\right)^{1/2}
&=\left(\sum_{\nu\in\mathcal J_\epsilon}\left\|\sum_{\mu=1}^{N_q}w_\mu(\widetilde W_\mu-W_\mu)\overline{\psi_\nu(\vect{p}_\mu)}+q_\nu\right\|_{\Frob}^2\right)^{1/2}\\
&\leq\left(\sum_{\nu\in\mathcal J_\epsilon}\left\|\sum_{\mu=1}^{N_q}w_\mu(\widetilde W_\mu-W_\mu)\overline{\psi_\nu(\vect{p}_\mu)}\right\|_{\Frob}^2\right)^{1/2}+\Delta_{\mathrm{quad}}\\
&\leq C_h\left(\sum_{\mu=1}^{N_q}w_\mu\|\widetilde W_\mu-W_\mu\|_{\Frob}^2\right)^{1/2}+\Delta_{\mathrm{quad}}\\
&\leq C_h\Delta_{\mathrm{off}}+\Delta_{\mathrm{quad}}+C_h\Delta_{\mathrm{data}}.
\end{align*}
Combining this estimate with the spectral cutoff separates the reconstruction error into off-grid error, quadrature error, data perturbation, and spectral truncation.

\begin{theorem}\label{thm:discrete_stability}
Suppose each $\Mt_\mu$ has full column rank and the quadrature weights are positive. Let $\Delta_{\mathrm{off}}$, $\Delta_{\mathrm{data}}$, and $\Delta_{\mathrm{quad}}$ be defined by \eqref{eq:discrete_error_components} and \eqref{eq:quadrature_error_norm}, and let $C_h=\|A\|_2$ with $A$ given by \eqref{eq:sampling_matrix_A}. Then the reconstruction \eqref{eq:discrete_reconstruction} satisfies
\begin{equation}\label{eq:combined_discrete_error_bound}
\|\widetilde{\Qt}-\Qt\|_{L^2(B(0,R);\C^{3\times3})}
\leq R^{3/2}\left[
\frac{C_h\Delta_{\mathrm{off}}+\Delta_{\mathrm{quad}}+C_h\Delta_{\mathrm{data}}}{\epsilon}
+\left(\sum_{\nu\notin\mathcal J_\epsilon}\|\tens{Q}_\nu\|_{\Frob}^2\right)^{1/2}
\right].
\end{equation}
\end{theorem}

\begin{proof}
On the retained modes, division by $|\alpha_\nu(C)|>\epsilon$ gives
\begin{equation*}
\left(\sum_{\nu\in\mathcal J_\epsilon}
\frac{\|\widetilde W_\nu-W_\nu\|_{\Frob}^2}{|\alpha_\nu(C)|^2}\right)^{1/2}
\leq\frac{C_h\Delta_{\mathrm{off}}+\Delta_{\mathrm{quad}}+C_h\Delta_{\mathrm{data}}}{\epsilon}.
\end{equation*}
The discarded modes contribute the truncation term, as in Theorem~\ref{thm:spectral_cutoff}. The triangle inequality gives the error bound on $B_1$. Using $\vect{x}=R\vect{z}$ and $\Qt_R(\vect{z})=\Qt(R\vect{z})$, we have
\begin{equation*}
\|\widetilde{\Qt}-\Qt\|_{L^2(B(0,R))}^2=R^3\|\widetilde{\Qt}(R\,\cdot)-\Qt_R\|_{L^2(B_1)}^2.
\end{equation*}
Taking square roots and applying the preceding estimate gives \eqref{eq:combined_discrete_error_bound}.
\end{proof}

\section{Numerical experiments}\label{sec:numerics}

We now examine the polarimetric recovery and spatial reconstruction through numerical experiments. We compute the conditioning of the polarimetric matrices and use analytically generated Born data to test noise amplification and off-grid error. Reconstruction experiments use far-field data generated by the full Maxwell volume integral equation (VIE), while the inversion is based on the Born approximation.

\subsection{Numerical configuration}\label{subsec:maxwell-gpswf-setup}

We assume all targets are supported in $B_1$. The dimension study uses a centered cube of half-side length $0.4$. The remaining imaging experiments use three rectangular blocks with centers and half-widths
\begin{equation}\label{eq:block_geometry}
\begin{array}{c|c|c}
r&\vect{c}_r&\vect{h}_r\\ \hline
1&(-0.26,0.25,0)&(0.16,0.16,0.16)\\
2&(0.25,0.25,0)&(0.15,0.15,0.16)\\
3&(0,-0.25,0)&(0.18,0.14,0.16).
\end{array}
\end{equation}
The material tensors are
\begin{align}
\Qt_1&=\begin{pmatrix}
0.20+0.040\ii&0.048\ii&-0.026\ii\\
-0.032\ii&0.12+0.030\ii&0.026\ii\\
0.034\ii&-0.014\ii&0.08+0.025\ii
\end{pmatrix},\nonumber\\
\Qt_2&=\begin{pmatrix}
0.09+0.030\ii&-0.020\ii&0.025\ii\\
0.030\ii&0.22+0.045\ii&0.041\ii\\
-0.011\ii&-0.029\ii&0.14+0.035\ii
\end{pmatrix},\label{eq:test_tensors}\\
\Qt_3&=\begin{pmatrix}
0.14+0.035\ii&0.021\ii&0.033\ii\\
-0.009\ii&0.07+0.025\ii&-0.018\ii\\
-0.023\ii&0.026\ii&0.19+0.050\ii
\end{pmatrix}.\nonumber
\end{align}
The centered cube carries $\Qt_1$, while the three blocks carry $\Qt_1$, $\Qt_2$, and $\Qt_3$, respectively. The contrast is zero outside the inclusions. These tensors are nonsymmetric and are not scalar multiples of a common matrix.

The VIE is discretized on a Cartesian voxel grid. In the $k=12$ full-tensor test, the forward discretization has $3695$ active voxel nodes and $11085$ field unknowns; $770$ distinct incident directions generate $1540$ polarized right-hand sides, and the largest sampled linear-system residual is $2.0\times10^{-16}$. At each target Fourier node, six configurations are selected from 24 nearby candidates according to distance, rank, and conditioning. Complex Gaussian noise is added to the stacked normalized far-field data and scaled to the prescribed relative level $\delta$. Unless stated otherwise, GPSWF modes satisfy the relative eigenvalue cutoff $|\alpha_\nu|>0.2\max_{\nu'}|\alpha_{\nu'}|$, subject to the prescribed dimension cap. 

Unless stated otherwise, the reconstructed images below show the real part of $Q_{11}$ on the cross-section $x_3=0$.

\subsection{Polarimetric conditioning and measurement redundancy}
\label{subsec:conditioning}

We first examine how the measurement configurations affect polarimetric conditioning. In particular, we compare the three- and six-configuration equiangular designs with the configurations selected for the tensor-reconstruction experiment following Section~\ref{sec:discretization}. The equiangular designs use the construction in Theorem~\ref{thm:three_configuration}, with $\theta_j=2\pi(j-1)/J$, $j=1,\ldots,J$, for $J=3,6$. For each equiangular design, the singular values depend only on the normalized Fourier radius $\rho$, since changing the direction of the Fourier point simply rotates the measurement configurations. As shown in Figure~\ref{fig:conditioning}, the two equiangular designs have coincident condition-number curves, while the six-configuration design has a larger smallest singular value. The condition numbers increase near both the center and boundary of the Fourier ball, consistent with Proposition~\ref{prop:center_boundary_conditioning}. The redundant configurations selected following Section~\ref{sec:discretization} generally yield larger smallest singular values than the three-configuration design and lie closer to the theoretical six-configuration curve.

\begin{figure}[ht]
\centering
\begin{minipage}{0.45\textwidth}
\centering
\includegraphics[width=\textwidth]{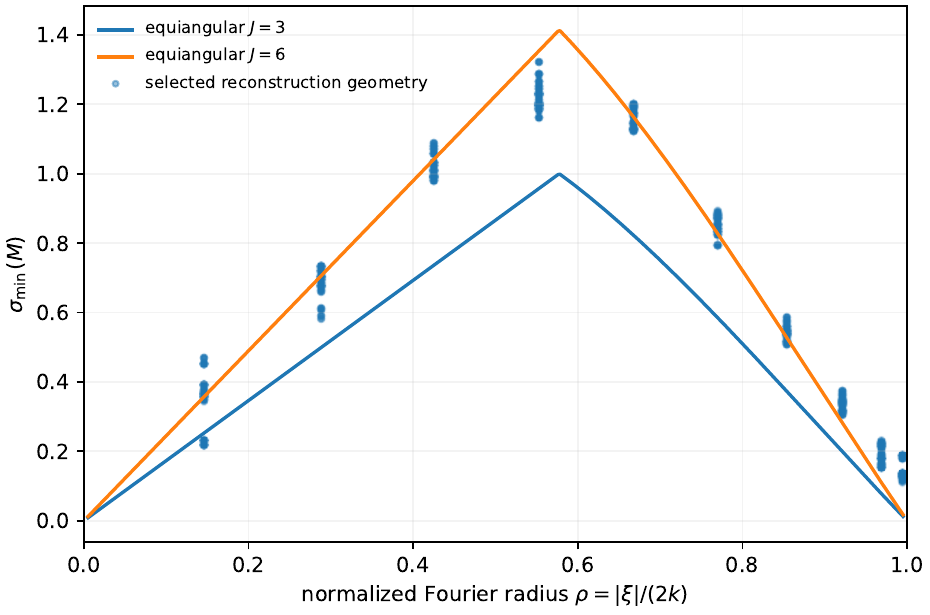}
\end{minipage}\hfill
\begin{minipage}{0.45\textwidth}
\centering
\includegraphics[width=\textwidth]{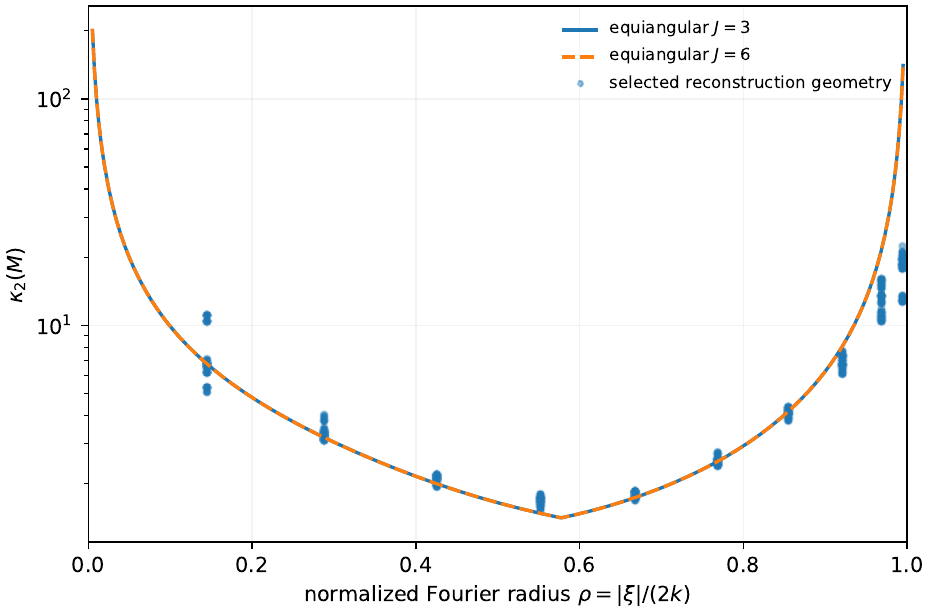}
\end{minipage}
\caption{Polarimetric conditioning versus normalized Fourier radius. Left: smallest singular value. Right: two-norm condition number. Curves show the three- and six-configuration equiangular designs; points represent the configurations selected for the $k=12$ tensor-reconstruction experiment.}
\label{fig:conditioning}
\end{figure}

\begin{figure}[ht]
\centering
\includegraphics[width=0.5\textwidth]{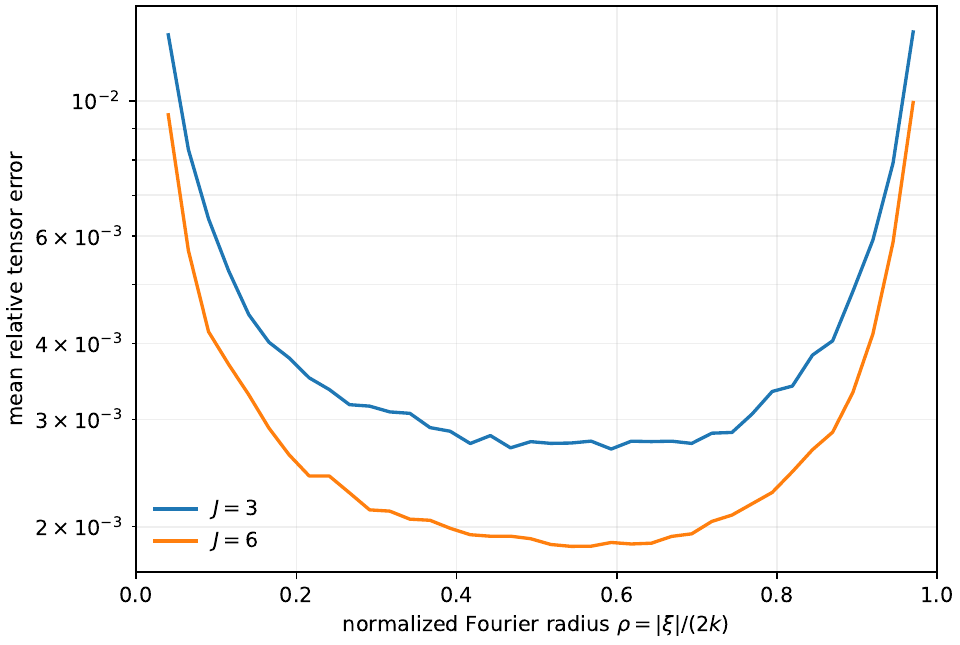}
\caption{Mean relative tensor-recovery error for three and six equiangular configurations, with the same noise variance per measured component.}
\label{fig:sparseviews}
\end{figure}


We next examine how measurement redundancy affects tensor recovery in the presence of noise. For a fixed complex tensor with unit Frobenius norm, we generate data from \eqref{eq:polarimetric_system} and add independent complex Gaussian noise with variance $10^{-6}$ per measured component for both designs. At each Fourier radius, we perform 200 noise trials and plot the mean relative Frobenius error in Figure~\ref{fig:sparseviews}. Six configurations give smaller mean errors than three, demonstrating the benefit of redundant measurements. For both designs, the errors increase near the center and boundary of the Fourier ball, where the smallest singular values decrease.

\subsection{Off-grid error}\label{subsec:mismatch}

We examine the error caused by using nearby realizable Fourier points in place of the target quadrature nodes. Using the analytic Fourier transform of the three-block target defined by \eqref{eq:block_geometry} and \eqref{eq:test_tensors}, we generate noise-free data through the Born relation \eqref{eq:born_tensor}. We vary the maximum normalized distance $h$ between the measurement and target Fourier points and compare the recovered tensor Fourier values with the exact values at the target nodes. Figure~\ref{fig:mismatch} shows log--log slopes of $1.03$, $1.01$, $0.95$, and $1.00$ for $\rho=0.25$, $0.50$, $0.75$, and $0.90$, respectively, in the small-distance regime, consistent with the first-order estimate in Proposition~\ref{prop:offgrid}.

\begin{figure}[ht]
\centering
\includegraphics[width=0.5\textwidth]{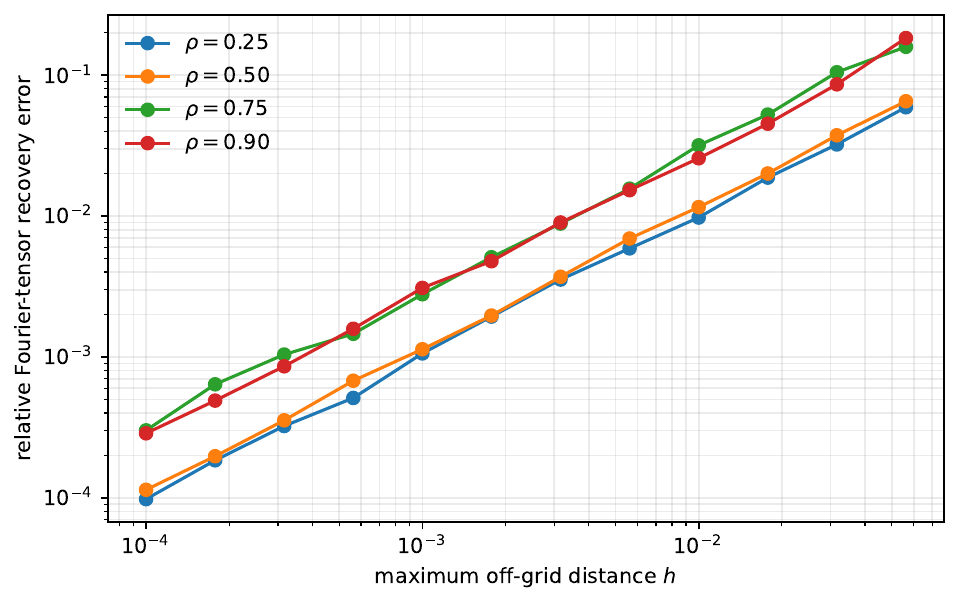}
\caption{Relative Fourier-tensor recovery error versus the maximum normalized off-grid distance $h$. The fitted log--log slopes are close to one, consistent with the estimate in Proposition~\ref{prop:offgrid}.}
\label{fig:mismatch}
\end{figure}

\begin{figure}[ht]
\centering
\includegraphics[width=0.9\textwidth]{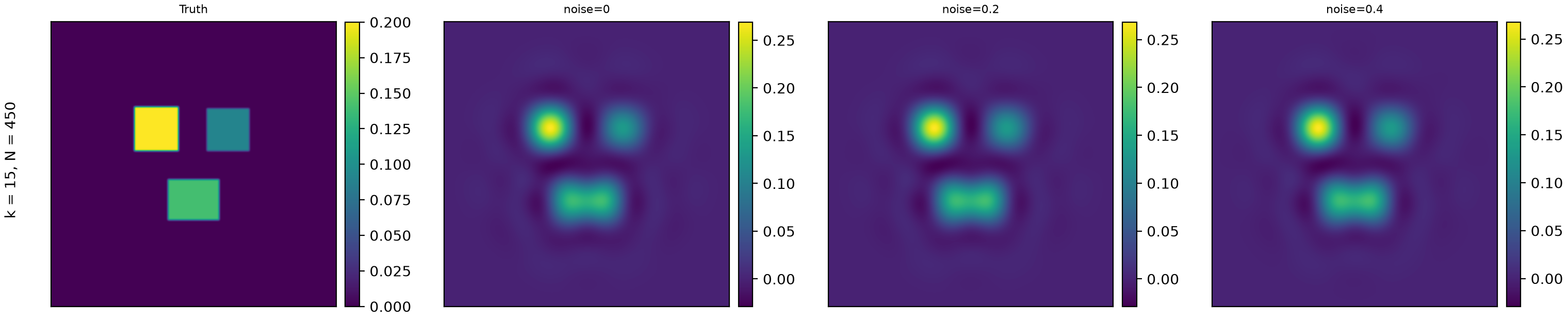}
\caption{Reconstruction of the $Q_{11}$ cross-section for relative far-field noise levels $\delta=0,0.2,0.4$, with $k=15$ and $N=450$. Each panel uses its own color scale.}
\label{fig:noise}
\end{figure}

\subsection{Retained dimension and noise}\label{subsec:rank_noise}

We first fix $N=450$ and use the three-block target at $k=15$ to examine the effect of far-field noise. The same Gaussian realization is rescaled to $\delta=0,0.2,0.4$, while the measurement configurations and reconstruction parameters remain unchanged. As shown in Figure~\ref{fig:noise}, the three blocks remain distinguishable across the tested noise levels, with relatively small changes in the reconstructed profiles.

\begin{figure}[ht]
\centering
\includegraphics[width=0.8\textwidth]{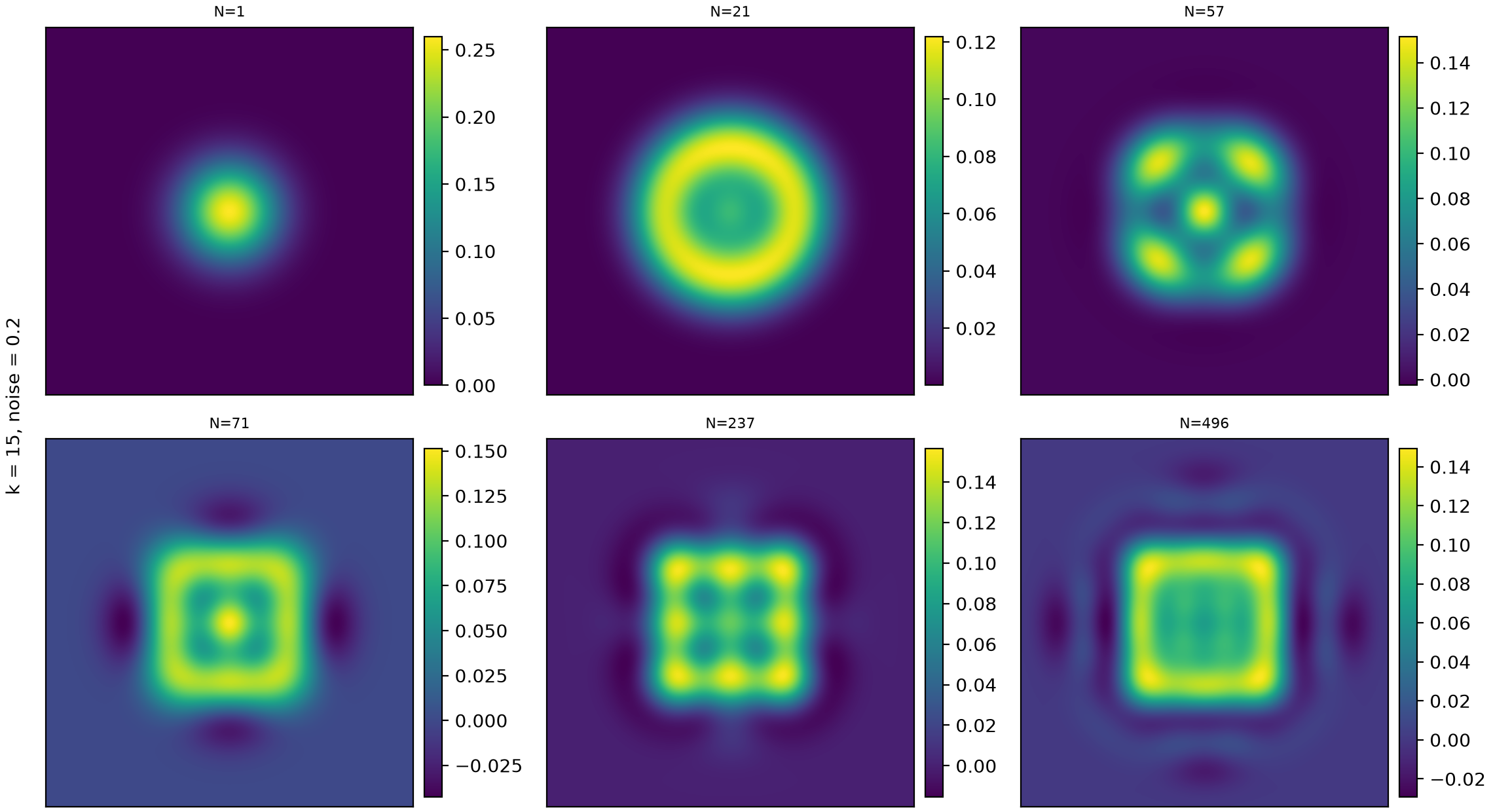}
\caption{Real part of the reconstructed $Q_{11}$ cross-section for increasing retained GPSWF dimension, with $k=15$ and $\delta=0.2$. Each panel uses its own color scale.}
\label{fig:dimension}
\end{figure}

We next examine the effect of the retained GPSWF dimension using the centered cube with $k=15$ and $\delta=0.2$. Figure~\ref{fig:dimension} shows reconstructions with $N=1,21,57,71,237,496$ modes. Small retained dimensions provide little geometric information, while larger dimensions gradually reveal the square cross-section. Further modes introduce finer details together with oscillations, illustrating the balance between truncation error and data-error amplification in \eqref{eq:cutoff_error}.

\begin{figure}[ht]
\centering
\includegraphics[width=0.8\textwidth]{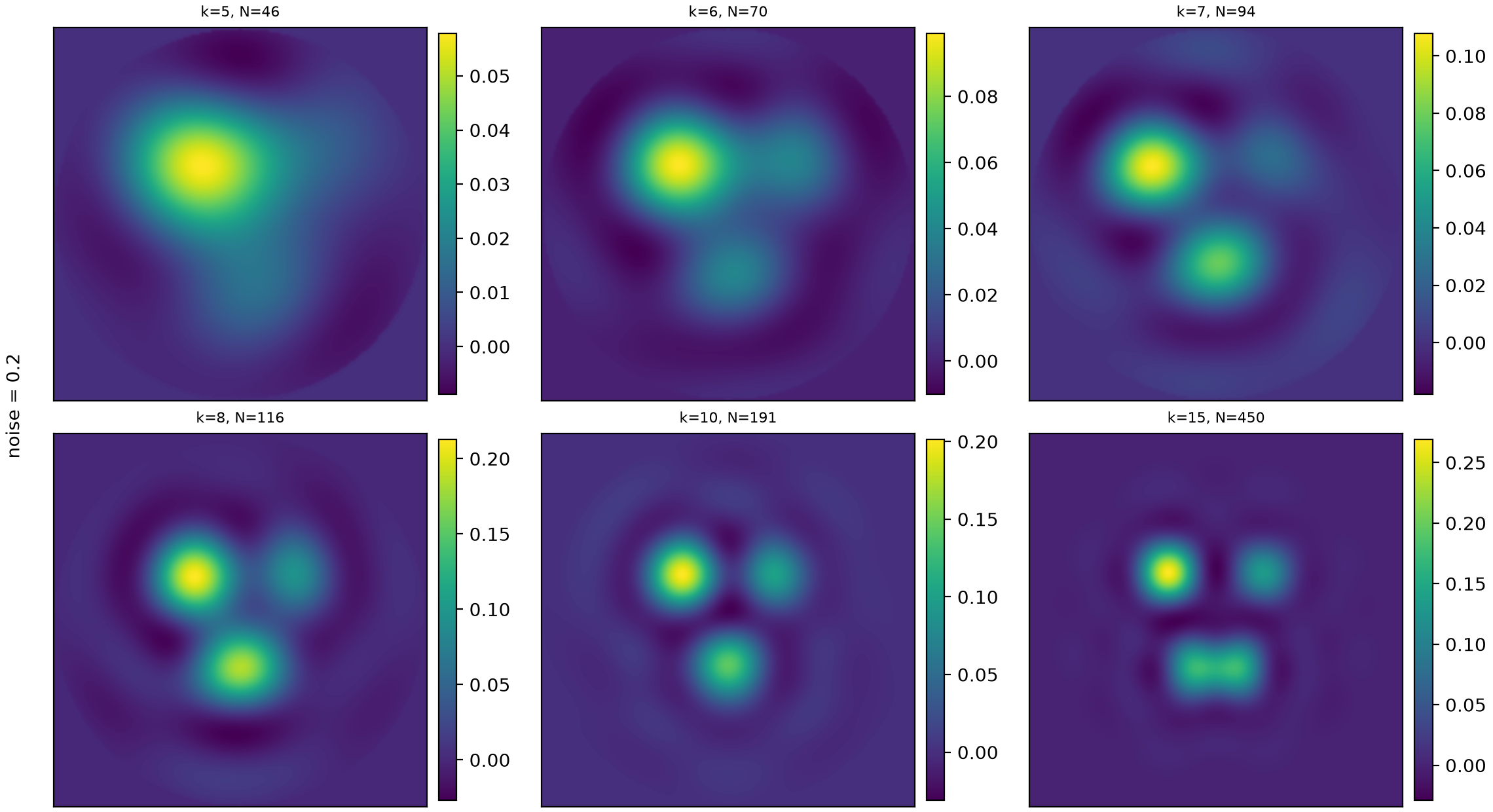}
\caption{Reconstruction of the $Q_{11}$ cross-section at different wavenumbers with $\delta=0.2$. Each panel uses its own color scale.}
\label{fig:frequency}
\end{figure}

\subsection{Effect of the wavenumber}\label{subsec:frequency}

We examine the three-block target at $k=5,6,7,8,10,15$ with $\delta=0.2$. The corresponding retained GPSWF dimensions are $46$, $70$, $94$, $116$, $191$, and $450$. Increasing $k$ expands the accessible Fourier ball and provides additional spatial information; the measurement and reconstruction discretizations are refined accordingly. Figure~\ref{fig:frequency} shows that the reconstructed profiles overlap at smaller wavenumbers, while the three blocks become distinguishable at $k=8$. Larger wavenumbers provide clearer separation and more geometric detail. The minimum boundary separation between the blocks is $0.20$, which is smaller than the half-wavelength $\pi/k\approx0.393$ at $k=8$. Nevertheless, the three blocks are distinguishable in this reconstruction.

\subsection{Comparison of reconstruction bases}\label{subsec:maxwell-basis-comparison}

We compare GPSWF, cube Fourier \cite{WangGuoZhangLiu2017}, and spherical-harmonic Bessel reconstructions using the same recovered Fourier samples and quadrature nodes at $k=8,12,15$ with $\delta=0.2$. The retained dimensions are $109$, $283$, and $450$ for GPSWF, and $130$, $339$, and $512$ for the other two bases, respectively. The comparison therefore uses identical input data with similar, but not identical, reconstruction dimensions.

\begin{figure}[ht]
\centering
\includegraphics[width=0.8\textwidth]{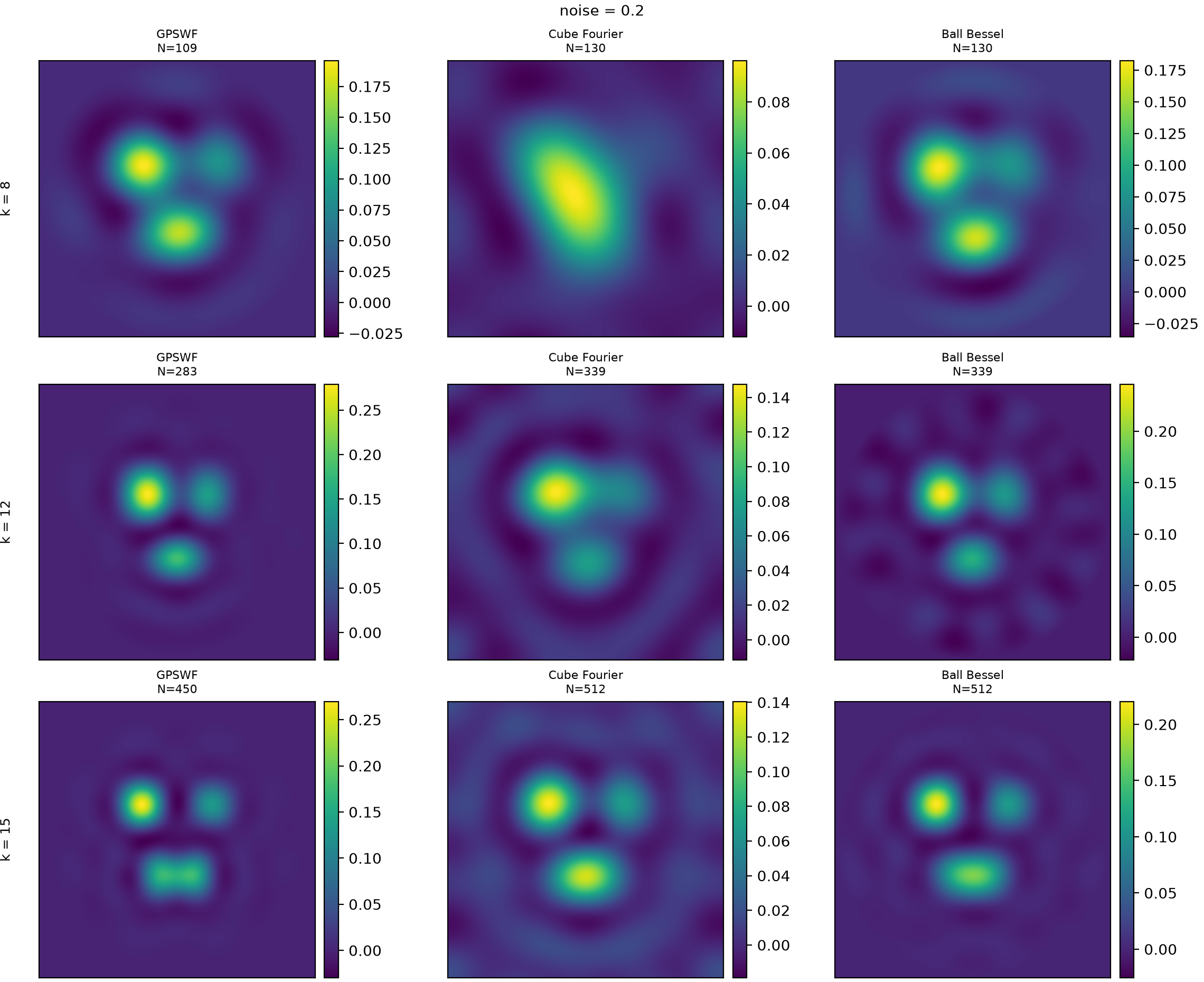}
\caption{GPSWF, cube Fourier, and spherical-harmonic Bessel reconstructions at $k=8,12,15$ with $\delta=0.2$. Each panel uses its own color scale.}
\label{fig:bases}
\end{figure}

Figure~\ref{fig:bases} shows that GPSWF and Bessel reconstructions distinguish the three blocks at $k=8$, whereas the Fourier reconstruction produces a largely merged profile. All three methods improve as the wavenumber increases. At $k=15$, the GPSWF reconstruction begins to reproduce the square cross-sections, while the other two remain more rounded. The Fourier and Bessel reconstructions also exhibit more pronounced approximately periodic oscillations and sidelobes. The weaker artifacts in the GPSWF reconstruction are consistent with its joint spatial--spectral concentration, which is adapted to both the bounded support and the accessible Fourier band.

\subsection{Full anisotropic tensor recovery}\label{subsec:tensor}

\begin{figure}[ht]
\centering
\includegraphics[width=0.8\textwidth]{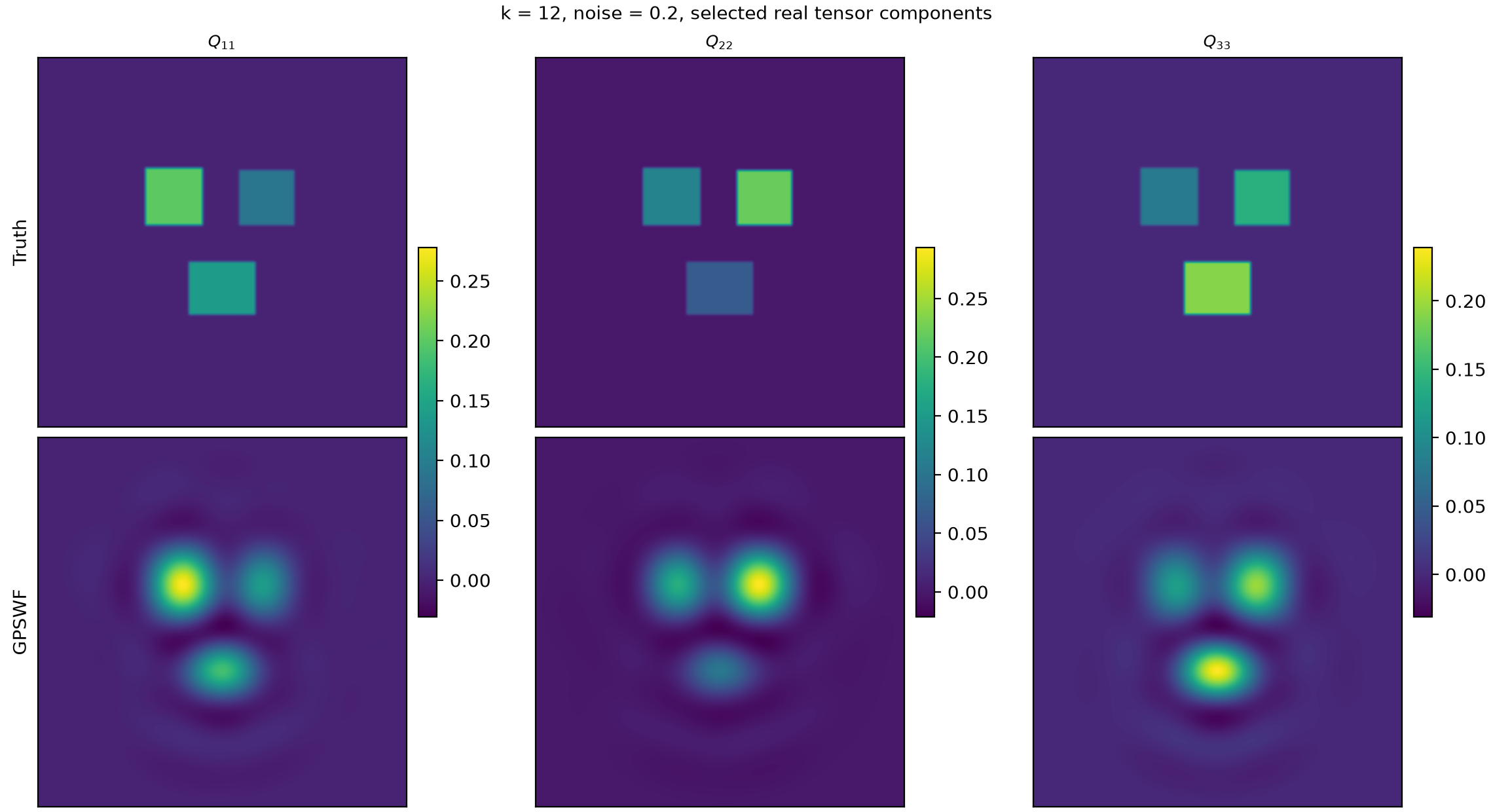}
\caption{Exact and reconstructed real parts of $Q_{11}$, $Q_{22}$, and $Q_{33}$ on $x_3=0$, with $k=12$, $\delta=0.2$, and $N=283$. Each exact/reconstructed pair shares a common color scale.}
\label{fig:tensorreal}
\end{figure}

\begin{figure}[ht]
\centering
\includegraphics[width=0.8\textwidth]{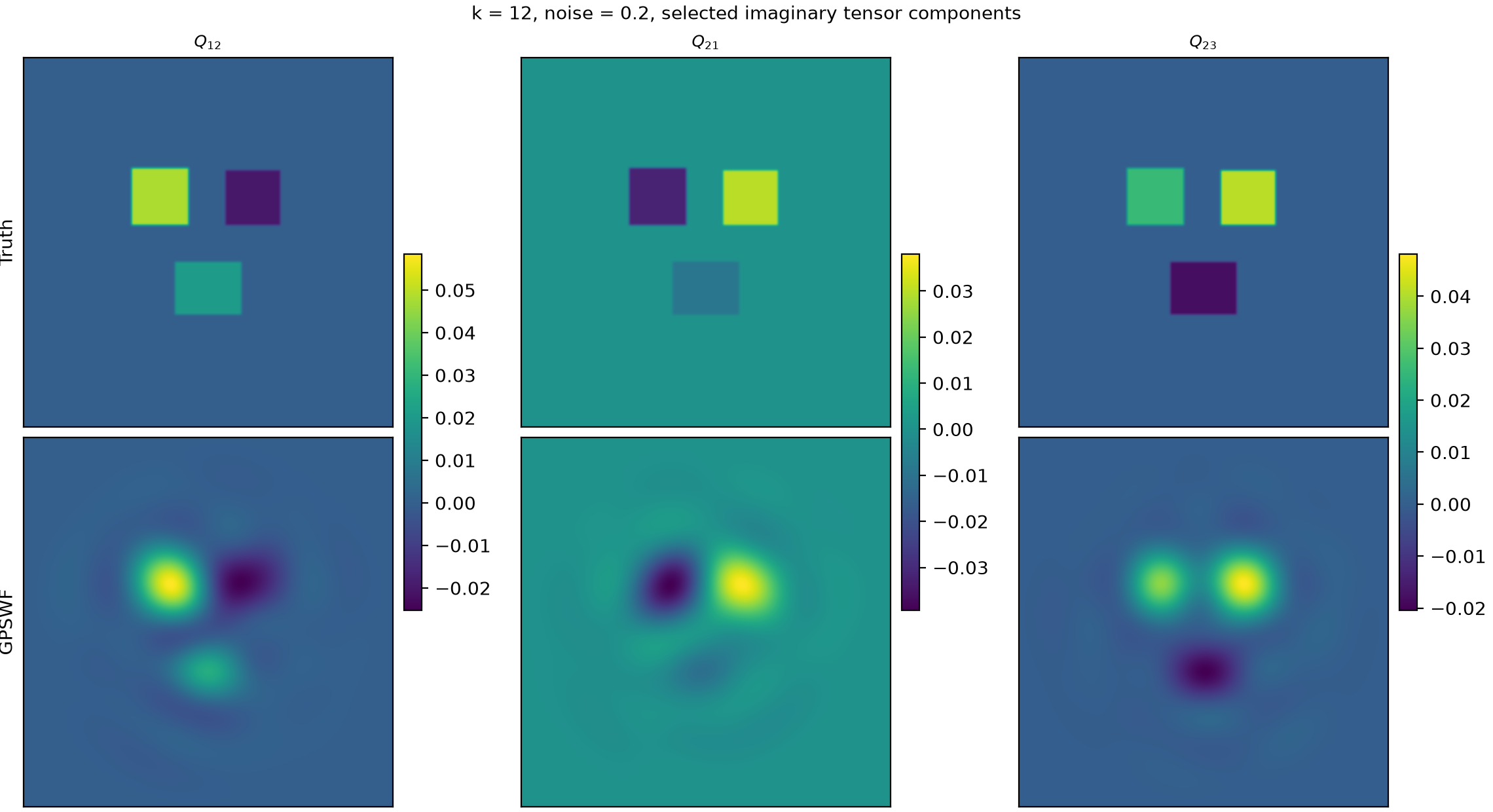}
\caption{Exact and reconstructed imaginary parts of $Q_{12}$, $Q_{21}$, and $Q_{23}$ on $x_3=0$ for the same experiment. Each exact/reconstructed pair shares a common color scale.}
\label{fig:tensorimag}
\end{figure}

The preceding reconstruction figures show only $Q_{11}$. We now examine the recovery of different components of the three-block tensor contrast, using $k=12$, $\delta=0.2$, and $N=283$. All nine complex entries are reconstructed simultaneously. The relative error of the complex tensor reconstruction over the target cross-sections is $0.451$, while the relative Frobenius errors of the block-averaged tensors are $0.244$, $0.226$, and $0.327$, respectively.

Figures~\ref{fig:tensorreal} and~\ref{fig:tensorimag} show the corresponding reconstructions of the diagonal real parts and three off-diagonal imaginary parts, respectively. The reconstructions distinguish the block locations and the different component amplitudes and signs, although the interfaces are smoothed and some amplitudes are distorted.

\section{Conclusion}\label{sec:conclusion}

We developed a conditioning-aware low-rank imaging framework for anisotropic Maxwell media. The main contribution lies in the polarimetric recovery of the tensor Fourier data. Maxwell transversality restricts the measurements to projected actions of the tensor Fourier transform. For a general $3\times3$ contrast, we proved that three configurations, each using two orthogonal incident polarizations, are necessary and provided an explicit three-configuration design that is sufficient at every nonzero interior Fourier point. The Gram determinant and singular-value asymptotics quantify the conditioning of this design and identify the intrinsic loss of full-tensor information at the boundary of the accessible Fourier ball.

After polarimetric recovery, the spatial inverse problem is the same compact restricted-Fourier inversion for every tensor component. Ball GPSWFs diagonalize this operator and provide a direct spectral cutoff. For finite acquisition, we quantified the error caused by replacing target Fourier nodes with nearby realizable measurement points. The resulting estimate separates off-grid error, quadrature error, measurement and Born modeling errors, and spectral truncation, while accounting for polarimetric and spectral amplification.

The numerical results support these mechanisms through analytically generated Born data and full-Maxwell VIE data. Three configurations are sufficient, while six reduce the recovery error under the same noise variance per measured component. The off-grid error exhibits first-order dependence on the node distance, and the configurations selected for the reconstruction remain full rank while becoming less well conditioned near the Fourier-ball boundary. The reconstruction distinguishes the material tensors from full-Maxwell data, although appreciable pointwise errors remain. The present analysis suggests two extensions: acquisition designs that improve the smallest singular value under limited-aperture constraints, and nonlinear iterations initialized by the present low-rank reconstruction.

\section*{Acknowledgments}
This work was partially supported by the National Key R\&D Program of China grant number 2024YFA1016000 and the Fundamental Research Funds for the Central Universities grant number 226-2025-00192.

\section*{Data Availability Statement}
Data and codes supporting the findings of this work are available upon request.

\section*{Conflict of interest} 
The authors have no conflicts of interest to declare. 

\appendix

\section{Gram matrix calculation for the three-configuration design}
\label{app:three_configuration}

We give the matrix calculation used in Theorem~\ref{thm:three_configuration}. Choose the standard Frobenius tensor basis. Columnwise vectorization yields
\begin{equation*}
\Mt_j=\tens{E}_j^\top\otimes\tens{P}_{\hat{\vect{x}}_j}.
\end{equation*}
Since $\tens{E}_j\tens{E}_j^*=\tens{P}_{\vect{d}_j}$ and the projection matrices are real,
\begin{equation*}
\Mt_j^*\Mt_j
=(\overline{\tens{E}_j}\tens{E}_j^\top)\otimes\tens{P}_{\hat{\vect{x}}_j}
=\tens{P}_{\vect{d}_j}\otimes\tens{P}_{\hat{\vect{x}}_j}.
\end{equation*}
Thus
\begin{equation*}
\Mt^*\Mt=\sum_{j=1}^{3}\tens{P}_{\vect{d}_j}\otimes\tens{P}_{\hat{\vect{x}}_j}.
\end{equation*}

Take $\vect{u},\vect{v},\vect{n}$ as the coordinate axes and write $t=\sqrt{1-\rho^2}$. The three transverse directions are
\begin{equation*}
\vect{s}_1=(1,0,0)^\top,\qquad
\vect{s}_2=(-1/2,\sqrt3/2,0)^\top,\qquad
\vect{s}_3=(-1/2,-\sqrt3/2,0)^\top.
\end{equation*}
Writing $\vect{s}_j=(a_j,b_j,0)^\top$, the projection matrices become
\begin{equation*}
\tens{P}_{\vect{d}_j}=
\begin{pmatrix}
1-t^2a_j^2&-t^2a_jb_j&-\rho ta_j\\
-t^2a_jb_j&1-t^2b_j^2&-\rho tb_j\\
-\rho ta_j&-\rho tb_j&t^2
\end{pmatrix},
\qquad
\tens{P}_{\hat{\vect{x}}_j}=
\begin{pmatrix}
1-t^2a_j^2&-t^2a_jb_j&\rho ta_j\\
-t^2a_jb_j&1-t^2b_j^2&\rho tb_j\\
\rho ta_j&\rho tb_j&t^2
\end{pmatrix}.
\end{equation*}

The sums needed to assemble the Gram matrix are
\begin{equation*}
\begin{aligned}
&\sum_j a_j=\sum_j b_j=\sum_j a_jb_j=0,
\qquad \sum_j a_j^2=\sum_j b_j^2=\frac32,\\
&\sum_j a_j^3=\frac34,\qquad
\sum_j a_jb_j^2=-\frac34,\qquad
\sum_j a_j^2b_j=\sum_j b_j^3=0,\\
&\sum_j a_j^4=\sum_j b_j^4=\frac98,\qquad
\sum_j a_j^2b_j^2=\frac38,\qquad
\sum_j a_j^3b_j=\sum_j a_jb_j^3=0.
\end{aligned}
\end{equation*}
For example,
\begin{equation*}
\begin{aligned}
[\Mt^*\Mt]_{11}
&=\sum_j(1-t^2a_j^2)^2
=3-3t^2+\frac98t^4,\\
[\Mt^*\Mt]_{13}
&=\sum_j(1-t^2a_j^2)\rho ta_j
=-\frac34\rho t^3.
\end{aligned}
\end{equation*}
Evaluating the remaining entries in the same way gives
\begin{equation*}
\Mt^*\Mt=
\left(
\begin{array}{ccc|ccc|ccc}
a&0&-d&0&c&0&d&0&-f\\
0&b&0&c&0&d&0&-d&0\\
-d&0&e&0&d&0&-f&0&0\\ \hline
0&c&0&b&0&d&0&-d&0\\
c&0&d&0&a&0&-d&0&-f\\
0&d&0&d&0&e&0&-f&0\\ \hline
d&0&-f&0&-d&0&e&0&0\\
0&-d&0&-d&0&-f&0&e&0\\
-f&0&0&0&-f&0&0&0&3t^4
\end{array}
\right),
\end{equation*}
where
\begin{equation*}
\begin{aligned}
a&=3-3t^2+\frac98t^4,&
b&=3-3t^2+\frac38t^4,&
c&=\frac38t^4,\\
d&=\frac34\rho t^3,&
e&=3t^2-\frac32t^4,&
f&=\frac32\rho^2t^2.
\end{aligned}
\end{equation*}

Let $\tens{F}_{ab}$ denote the matrix with a single unit entry in position $(a,b)$. Consider the following Frobenius-orthonormal tensor basis:
\begin{equation*}
\begin{array}{ccc}
\dfrac{\tens{F}_{11}+\tens{F}_{22}}{\sqrt2},&
\tens{F}_{33},&
\dfrac{\tens{F}_{11}-\tens{F}_{22}}{\sqrt2},\\[6pt]
\dfrac{\tens{F}_{13}-\tens{F}_{31}}{\sqrt2},&
\dfrac{\tens{F}_{12}+\tens{F}_{21}}{\sqrt2},&
\dfrac{\tens{F}_{23}-\tens{F}_{32}}{\sqrt2},\\[6pt]
\dfrac{\tens{F}_{12}-\tens{F}_{21}}{\sqrt2},&
\dfrac{\tens{F}_{13}+\tens{F}_{31}}{\sqrt2},&
\dfrac{\tens{F}_{23}+\tens{F}_{32}}{\sqrt2}.
\end{array}
\end{equation*}
Let $\tens{S}$ have as its columns the columnwise vectorizations of these matrices, read from left to right and top to bottom. Then $\tens{S}^*\tens{S}=I_9$.

For the first two basis matrices, the corresponding Gram entries are
\begin{equation*}
\begin{aligned}
[\tens{S}^*\Mt^*\Mt\tens{S}]_{11}
&=\frac12\bigl([\Mt^*\Mt]_{11}+[\Mt^*\Mt]_{15}
+[\Mt^*\Mt]_{51}+[\Mt^*\Mt]_{55}\bigr)
=a+c,\\
[\tens{S}^*\Mt^*\Mt\tens{S}]_{12}
&=\frac{1}{\sqrt2}\bigl([\Mt^*\Mt]_{19}+[\Mt^*\Mt]_{59}\bigr)
=-\sqrt2f.
\end{aligned}
\end{equation*}
For the third and fourth basis matrices,
\begin{equation*}
\begin{aligned}
[\tens{S}^*\Mt^*\Mt\tens{S}]_{33}&=a-c,\\
[\tens{S}^*\Mt^*\Mt\tens{S}]_{34}
&=\frac12\bigl([\Mt^*\Mt]_{17}-[\Mt^*\Mt]_{13}
-[\Mt^*\Mt]_{57}+[\Mt^*\Mt]_{53}\bigr)
=2d,\\
[\tens{S}^*\Mt^*\Mt\tens{S}]_{44}&=e+f.
\end{aligned}
\end{equation*}
The remaining products give the block decomposition
\begin{equation}\label{eq:appendix_gram_blocks}
\tens{S}^*(\Mt^*\Mt)\tens{S}
=\operatorname{diag}\!\left(
\tens{B}_0,\tens{B}_+,\tens{B}_-,b-c,e-f,e-f
\right),
\end{equation}
where
\begin{equation*}
\tens{B}_0=
\begin{pmatrix}
a+c&-\sqrt2f\\
-\sqrt2f&3t^4
\end{pmatrix},
\qquad
\tens{B}_+=
\begin{pmatrix}
a-c&2d\\
2d&e+f
\end{pmatrix},
\qquad
\tens{B}_-=
\begin{pmatrix}
b+c&-2d\\
-2d&e+f
\end{pmatrix}.
\end{equation*}
Using $t^2=1-\rho^2$, we have
\begin{equation*}
\begin{aligned}
a+c&=\frac32(1+\rho^4),&
a-c=b+c&=\frac34(1+\rho^2)^2,\\
e+f&=\frac32t^2(1+2\rho^2),&
b-c&=3\rho^2,\qquad e-f=\frac32t^2.
\end{aligned}
\end{equation*}
Consequently,
\begin{equation*}
\tens{B}_0=
\begin{pmatrix}
\frac32(1+\rho^4)&-\frac{3\sqrt2}{2}\rho^2t^2\\
-\frac{3\sqrt2}{2}\rho^2t^2&3t^4
\end{pmatrix},
\qquad
\tens{B}_\pm=
\begin{pmatrix}
\frac34(1+\rho^2)^2&\pm\frac32\rho t^3\\
\pm\frac32\rho t^3&\frac32t^2(1+2\rho^2)
\end{pmatrix}.
\end{equation*}

The determinants of these blocks are
\begin{align*}
\det\tens{B}_0
&=\frac92(1+\rho^4)t^4-\frac92\rho^4t^4
=\frac92t^4,\\
\det\tens{B}_\pm
&=\frac98t^2\left[(1+\rho^2)^2(1+2\rho^2)-2\rho^2t^4\right]\\
&=\frac98t^2(9\rho^4+2\rho^2+1).
\end{align*}
Since $\tens{S}$ is unitary, \eqref{eq:appendix_gram_blocks} gives
\begin{align*}
\det(\Mt^*\Mt)
&=3\rho^2\left(\frac32t^2\right)^2
\det\tens{B}_0\,\det\tens{B}_+\,\det\tens{B}_-\\
&=3\rho^2\left(\frac32t^2\right)^2
\frac92t^4\left[\frac98t^2(9\rho^4+2\rho^2+1)\right]^2\\
&=\frac{3^9}{2^9}\rho^2(1-\rho^2)^6
(9\rho^4+2\rho^2+1)^2.
\end{align*}
This is positive for $0<\rho<1$, proving the full-rank assertion in Theorem~\ref{thm:three_configuration}.

\section{Computation of the ball GPSWFs}
\label{sec:gpswf_implementation}

This appendix describes the numerical construction of the ball GPSWFs used in Section~\ref{sec:gpswf}. The restricted Fourier operator $\F_C$ commutes with the differential operator \cite{Greengard2024,ZhangLiWangZhang2020}
\begin{equation*}
\Dop_C=-\nabla\cdot\big((\Id-\vect{x}\vect{x}^{\top})\nabla\big)+C^2|\vect{x}|^2.
\label{eq:ball_operator}
\end{equation*}
The eigenfunctions of $\Dop_C$ can therefore be chosen to be the GPSWFs of $\F_C$.

Writing
\begin{equation*}
\psi_{\ell n m}(r,\omega;C)=R_{\ell n}(r;C)Y_\ell^m(\omega),
\end{equation*}
and separating the angular variables gives the radial equation
\begin{equation*}
\left[-(1-r^2)\frac{\mathrm d^2}{\mathrm dr^2}-\left(\frac{2}{r}-4r\right)\frac{\mathrm d}{\mathrm dr}+\frac{\ell(\ell+1)}{r^2}+C^2r^2\right]R_{\ell n}(r;C)=\chi_{\ell n}(C)R_{\ell n}(r;C).
\label{eq:radial_differential}
\end{equation*}
Here $\chi_{\ell n}(C)$ is the corresponding differential eigenvalue. We take the solution that is regular at the origin and belongs to the finite-energy domain at $r=1$.

To discretize this equation, set $\zeta=2r^2-1$ and $v=\ell+1/2$. Let $\widetilde P_j^{(0,v)}$ denote the Jacobi polynomials orthonormal with respect to the weight $(1+\zeta)^v$ on $(-1,1)$. The radial function is approximated by
\begin{equation}
R_{\ell n}(r;C)\approx 2^{\ell/2+5/4}r^\ell\sum_{j=0}^{K}\beta_j^{\ell n}(C)\widetilde P_j^{(0,v)}(2r^2-1).
\label{eq:jacobi_expansion}
\end{equation}
With this normalization, the radial $L^2((0,1),r^2\dd r)$ norm of $R_{\ell n}$ equals the Euclidean norm of the coefficient vector $\beta^{\ell n}$.

For $C=0$, the $j$th basis function has the differential eigenvalue
\begin{equation*}
\gamma_{\ell+2j}=(\ell+2j)(\ell+2j+3).
\end{equation*}
The orthonormal Jacobi recurrence coefficients are
\begin{equation*}
a_j=\frac{2(j+1)(j+v+1)}{(2j+v+2)\sqrt{(2j+v+1)(2j+v+3)}},\qquad b_j=\frac{v^2}{(2j+v)(2j+v+2)}.
\end{equation*}
Since $r^2=(1+\zeta)/2$, truncation at degree $K$ reduces the radial differential problem to
\begin{equation*}
A_\ell\beta^{\ell n}=\chi_{\ell n}\beta^{\ell n},
\label{eq:tridiag}
\end{equation*}
where
\begin{equation*}
(A_\ell)_{jj}=\gamma_{\ell+2j}+\frac{C^2}{2}(1+b_j),\qquad (A_\ell)_{j,j+1}=(A_\ell)_{j+1,j}=\frac{C^2}{2}a_j.
\end{equation*}
Thus, for each angular degree $\ell$, the radial functions are obtained from a real symmetric tridiagonal eigenvalue problem. The eigenvectors are normalized to unit Euclidean norm and substituted into \eqref{eq:jacobi_expansion}.

It remains to compute the Fourier eigenvalue associated with each radial function. Substituting the separated form \eqref{eq:gpswf_separation} into the eigenvalue equation \eqref{eq:gpswf_eigen} and applying the plane-wave expansion gives
\begin{equation}
4\pi\ii^\ell\int_0^1j_\ell(Crs)R_{\ell n}(s;C)s^2\dd s
=\alpha_{\ell n}(C)R_{\ell n}(r;C),
\label{eq:radial_integral}
\end{equation}
where $j_\ell$ is the spherical Bessel function. Define
\begin{equation*}
(\mathcal K_{\ell,C}R)(r)=\int_0^1j_\ell(Crs)R(s)s^2\dd s.
\end{equation*}
Since $R_{\ell n}$ is normalized in $L^2((0,1),r^2\dd r)$, taking the radial inner product of \eqref{eq:radial_integral} with $R_{\ell n}$ gives
\begin{equation*}
\alpha_{\ell n}(C)=4\pi\ii^\ell\int_0^1(\mathcal K_{\ell,C}R_{\ell n})(r)\overline{R_{\ell n}(r;C)}r^2\dd r.
\label{eq:fourier_eigenvalue_computation}
\end{equation*}
The phase $\ii^\ell$ corresponds to the positive exponential in \eqref{eq:restricted_fourier}. For the negative-exponent Fourier convention, the corresponding eigenvalue is its complex conjugate.

\bibliographystyle{plain}
\bibliography{ref}
\end{document}